\documentclass[11pt]{amsart}
\usepackage{url}
\usepackage{hyperref}
\usepackage{breakurl}
\usepackage{graphicx}
\usepackage{fancyhdr}
\usepackage[all]{xypic}
\usepackage{mathrsfs}
\input xypic
\usepackage{amsmath,amssymb,amsfonts,amsthm}
\usepackage[top=3.0cm, bottom=3.0cm, left=3.0cm, right=3.0cm]{geometry}
\newtheorem{thm}{Theorem}[section]
\newtheorem{theorem}[thm]{Theorem}
\newtheorem {definition}[thm]{Definition}

\newtheorem{corollary}[thm]{Corollary}
\newtheorem{lemma}[thm]{Lemma}

\newtheorem{proposition}[thm]{Proposition}
\newtheorem{example}[thm]{Example}
\newtheorem{remark}[thm]{Remark}

\newtheorem {conjecture}[thm]{Conjecture}

\numberwithin{equation}{section}
\begin{document}
\newpage
\thispagestyle{empty}

\baselineskip=16pt
\title{Right simple singularities in positive characteristic}      
\author{Gert-Martin Greuel and Nguyen Hong Duc}
\address{Gert-Martin Greuel\newline \indent Universit\"{a}t Kaiserslautern, Fachbereich Mathematik, Erwin-Schr\"{o}dinger-Strasse, 
\newline \indent 67663 Kaiserslautern}
\email{greuel@mathematik.uni-kl.de}
\address{Nguyen Hong Duc
\newline\indent Institute of Mathematics, 18 Hoang Quoc Viet Road, Cau Giay District \newline \indent  10307, Hanoi.} 
\email{nhduc@math.ac.vn}

\address{Universit\"{a}t Kaiserslautern, Fachbereich Mathematik, Erwin-Schr\"{o}dinger-Strasse,  
\newline \indent 67663 Kaiserslautern}
\email{dnguyen@mathematik.uni-kl.de}

\date{\today}                  
\maketitle
\begin{abstract}
We classify isolated singularities $f\in K[[x_1,\ldots, x_n]]$, which are simple, i.e. have no moduli, w.r.t. right equivalence, where $K$ is an algebraically closed field of characteristic $p>0$. For $K=\mathbb R$ or $\mathbb C$ this classification was initiated by Arnol'd, resulting in the famous ADE-series. The classification w.r.t. contact equivalence for $p>0$ was done by Greuel and Kr\" oning with a result similar to Arnol'd's. It is surprising that w.r.t. right equivalence and any given $p>0$ we have only finitely many simple singularities, i.e. there are only finitely many $k$ such that $A_k$ and $D_k$ are right simple, all the others have moduli. We conjecture a similar finiteness result for singularities with an arbitrary number of moduli. A major point of this paper is the generalization of the notion of modality to the algebraic setting, its behaviour under morphisms, and its relations to formal deformation theory. As an application we show that the modality is semicontinuous in any characteristic.
\end{abstract}

\section{Introduction}
We classify isolated singularities $f\in K[[x_1,\ldots, x_n]]$, $K$ an algebraically closed field of characteristic $p>0$, which have no moduli (modality 0) w.r.t. right equivalence, meaning that there are only finitely many right equivalence classes (see Definition \ref{def1.4.4.1}), where $f$ and $g$ are right equivalent, if they differ by a change of coordinates, see Section \ref{sec.a2}. These singularities are called right simple, following Arnol'd, who classified right simple singularities for $K=\mathbb R$ and $\mathbb C$ (cf. \cite{Arn72}). He showed that the simple singularities are exactly the ADE-singularities, i.e. the two infinite series $A_k, k\geq 1$, $D_k, k\geq 4$, and the three exceptional singularities $E_6, E_7, E_8$. It turned out later that the ADE-singularities of Arnol'd are also exactly those of modality 0 for contact equivalence. In the late eighties, Greuel and Kr\" oning showed in \cite{GK90} that the contact simple singularities over a field of positive characteristic are again exactly the ADE-singularities but with a few more normal forms in small characteristic.

A classification w.r.t. right equivalence in positive characteristic however, was never considered so far. A surprising fact of our classification is that for any fixed $p>0$ there exist only finitely many right simple singularities. We conjecture that this is a general fact for right equivalence in positive characteristic (cf. Conjecture \ref{conj}). For example, if $p=2$ and $n$ is even, there is just one right simple hypersurface,
$$x_1x_2+x_3x_4+\ldots+x_{n-1}x_n,$$
while for $n$ odd no right simple singularity exist. A table with normal forms for any $n\geq 1$ and any $p>0$ is given in section 3 (Theorems \ref{thm1.5.0} - \ref{thm1.5.2}). The problem is even interesting for univariate power series ($n=1$) (see Section \ref{sec3.1}). 
  
In section 2 we give a precise definition of the number of moduli (modality) for families of power series parametrized by an algebraic variety. In fact, we give two definitions of $G$-modality, both related to the action of an algebraic group $G$ on a variety $X$ and show that they coincide (Propositions \ref{pro1.4.0} and \ref{pro1.4.2}), a result which is valid in any characteristic. This unifies the arguments used in the classification, avoiding a lot of similar calculations for different cases. 

Moreover, we prove that the $G$-modality is upper semicontinuous for $G$ the right resp. the contact group (Proposition \ref{pro1.4.5.3}).

We introduce the notion of $G$-completeness (Definition \ref{def1.4.4.3}) which suffices to determine the modality, and we generalize the Kas-Schlessinger theorem \cite{KaS72} to deformations (unfoldings) of formal power series over algebraic varieties. The semiuniversal deformation with section of an isolated hypersurface singularity is $G$-complete for $G$ the right resp. the contact group (see Proposition \ref{pro1.4.4.1}). However, in contrast to the complex analytic case, the usual semiuniversal deformation is not sufficient to determine the modality and hence is not $G$-complete; we have to consider versal deformations with section (cf. Example \ref{rm1.4.1}).
\subsection*{Acknowledgement} 
We would like to thank the referees for their careful reading of the manuscript an helpful comments which improved the presentation of this paper. The second author was partly supported by DAAD (Germany), NAFOSTED (Vietnam), and the OWLF programme of the Mathematisches Forschungsinstitut Oberwolfach (Germany).
\section{Modality}\label{sec2}
In the sixties V. I. Arnol'd introduced the notion of modality into singularity theory for real and complex hypersurfaces (cf. \cite[Part II]{AGV85}), related to Riemann's idea of moduli for Riemann surfaces. The purpose of this section is to make the notion of modality precise in the case of hypersurface singularities over an algebraically closed field $K$ of arbitrary characteristic and relate it to deformation theory. We investigate right (resp. contact) unfoldings of a formal power series over algebraic varieties and define the modality w.r.t. unfoldings. We introduce the notion of right (resp. contact) complete unfoldings, which can be used to give an alternative definition of right (resp. contact) modality. We use \'etale neighbourhoods in order to show that an algebraic representation of the semiuniversal deformation is complete, see Proposition \ref{pro1.4.4.1}. The results of this section are used for the classification in Section 3.
\subsection{$G$-modality}
We use a Rosenlicht stratification of a variety\footnote{By an algebraic variety we mean a separated scheme of finite type over an algebraically closed field $K$, see \cite{Har77}, which is fixed through this paper. By a point we mean a closed point.} $X$ under the action of an algebraic group $G$ to define modality. By Rosenlicht \cite[Thm.2]{Ros56} (see also \cite{Ros63}) there exists an open dense subset $X_1\subset X$, which is invariant under $G$ s.t. $X_1/G$ is a geometric quotient (cf. \cite[\S 1]{MFK82}), in particular, the orbit space $X_1/G$ is an algebraic variety and the projection $p_1\colon X_1\to X_1/G,\ x\mapsto [x]$, is a surjective morphism. 

Since $X\setminus X_1$ is a variety of lower dimension, which is invariant under $G$, we can apply the theorem of Rosenlicht to $X\setminus X_1$ and get an invariant open dense subset $X_2\subset X\setminus X_1$ s.t. $X_2/G$ is a geometric quotient. Continuing in this way with $X_3\subset (X\setminus X_1)\setminus X_2$, we can finally write $X$ as finite disjoint union of $G$-invariant locally closed algebraic subvarieties $X_{i}, i=1,\ldots,s$, such that $X_{i}/G$ is a geometric quotient with quotient morphism $p_{i}\colon X_{i}\to X_{i}/G$. We call $\{X_{i}, i=1,\ldots,s\}$ a {\em Rosenlicht stratification of $X$ under G}. Note that a Rosenlicht stratification is by no means unique and that the proof of Rosenlicht, which works for arbitrary $G$, is not constructive.

\begin{definition}\label{def1.4.1}{\rm 
Let $\{X_{i}, i=1,\ldots,s\}$ be a Rosenlicht stratification of the algebraic variety $X$ under the action of an algebraic group $G$ with quotient morphisms $p_i: X_i\to X_i/G$, and let $U$ be an open neighbourhood of $x\in X$. We define
$$G\text{-}\mathrm{mod}(U):=\max_{1\leq i\leq s}\{\dim \big (p_i(U\cap X_i)\big)\},$$
and call 
$$G\text{-}\mathrm{mod}(x):=\min \{G\text{-}\mathrm{mod}(U)\ |\ U \text{ a neighbourhood of x}\}$$
the {\em $G$-modality} of $x$ (in $X$).
}\end{definition}

Note that here and later, wherever we write $\dim S$, the set $S$ is {\em constructible}, i.e. it is a finite union of locally closed subsets of a variety, so that $\dim S$ is defined. By Corollary \ref{cor1.4.1} $G\text{-}\mathrm{mod}(U)$ and $G\text{-}\mathrm{mod}(x)$ are independent of the Rosenlicht stratification.

\begin{remark}{\rm
Let $\{X_i\}$ be a Rosenlicht stratification of $X$ under $G$. In \cite{Wal83} Wall introduced the {\em $r$-modal set} $M_r(X)$ to be the closure of the union $\cup\{X_i| i\in I_r\}$ with $I_r:=\{i| \dim X_i/G\geq r\}$, which satisfies
$$X=M_0\supset M_1\supset \cdots\supset M_{\dim X}.$$
Using our definition of $G\text{-}\mathrm{mod}(x)$ one can show that
$$M_r(X)=\{x\in X| G\text{-}\mathrm{mod}(x)\geq r\}.$$
Moreover, we have $G\text{-}\mathrm{mod}(x)=r\Leftrightarrow x\in M_r(X)\setminus M_{r+1}(X)$ for $x\in X$, and for any open subset $U\subset X$, $G\text{-}\mathrm{mod}(U)=\max\{G\text{-}\mathrm{mod}(x)| x\in U\}$.
}\end{remark}

Now we will give the definition of modality for isolated singularities. We recall two important invariants of singularities. For $f\in K[[{\bf x}]]$, let $\mu(f):=\dim K[[{\bf x}]]/j(f)$ be the {\em Milnor number} and $\tau(f):=\dim K[[{\bf x}]]/(\langle f\rangle+j(f))$ the {\em Tjurina number} of $f$, where $j(f)$ is the jacobian ideal, i.e. the ideal in $K[[{\bf x}]]$ generated by all partials of $f$. Note that $f$ (resp. $K[[{\bf x}]]/\langle f\rangle$) has an {\em isolated} singularity if $\mu(f)<\infty$ (resp. $\tau(f)<\infty$).
$$$$

Let $f\in\mathfrak{m}\subset K[[{\bf x}]]$, with $\mathfrak{m}=\langle {\bf x}\rangle$ the maximal ideal, be such that $\mu(f)<\infty$ (resp. $\tau(f)<\infty$) and let $G=\mathcal{R}$ be right group (resp. $G=\mathcal{K}$ be the contact group). In order to define the modality of a power series $f$ by using algebraic groups, we have to consider its $k$-jet 
$$j^kf:=j^k(f):=\text{ image of }f \text{ in }J_k:=K[[{\bf x}]]/\mathfrak{m}^{k+1}$$
for sufficiently large $k$. Let $G_k=\mathcal{R}_k$ resp. $G_k=\mathcal{K}_k$ the $k$-jet of $G$, see Section \ref{sec.a2}.

\begin{definition}\label{def1.4.4.1}{\rm
We define the {\em $G$-modality} of $f$, $G\text{-}\mathrm{mod}(f)$, to be the $G_k$-modality of $j^k(f)$ in $J_k$ (in the sense of Definition \ref{def1.4.1}) for sufficiently large integer $k$. We call $f$ {\em right} (resp. {\em contact}) {\em simple, unimodal, bimodal} and {\em $r$-modal} if $\mathcal{R}$-$\mathrm{mod}(f)$ (resp. $\mathcal{K}$-$\mathrm{mod}(f)$) equals to 0, 1, 2 and $r$ respectively.
}\end{definition}

Here, an integer $k$ is {\em sufficiently large} for $f$ w.r.t. $G$ if there exists a neighbourhood $U$ of $j^k(f)$ in $J_k$ s.t. every $g\in \mathfrak{m}$ with $j^kg\in U$ is {\em $k$-determined w.r.t. $G$}. This means that each $h\in \mathfrak{m}$ s.t. $j^k(h)=j^k(g)$, is $G$-equivalent to $g$. Combining Propositions \ref{pro1.4.5.0} and \ref{pro1.4.5} below we obtain that $G$-$\mathrm{mod}(f)$ is independent of the sufficiently large $k$. The existence of a sufficiently large integer $k$ for $f$ w.r.t $G$ will be shown by the following proposition.

\begin{proposition}\label{pro1.4.3}
Let $f\in \mathfrak{m}^2$ be such that $\mu(f)<\infty$ (resp. $\tau(f)<\infty$) and let $G=\mathcal{R}$ (resp. $G=\mathcal{K}$). Then every $k\geq 2\cdot\mu(f)$ (resp. $k\geq 2\cdot\tau(f)$) is sufficiently large for $f$ w.r.t. $\mathcal{R}$ (resp. w.r.t. $\mathcal{K}$). For $f\in\mathfrak{m}\setminus \mathfrak{m}^2$, $k=1$ is sufficiently large for $f$ w.r.t. $G$.
\end{proposition}

\begin{proof}
By the upper semi-continuity of $\mu, \tau$ (Lemma \ref{pro1.1.1}), the subsets 
$$U_{\mu}:=\{g\in K[[{\bf x}]]\ |\ \mu(g)\leq \mu(f)\} \text{ and } U_{\tau}:=\{g\in K[[{\bf x}]]\ |\ \tau(g)\leq \tau(f)\}$$
are open\footnote{$V\subset K[[{\bf x}]]$ is open iff $j^l(V)$ is open in $J_l$ for all $l$.} in $K[[{\bf x}]]$. It follows from \cite[Cor. 1]{BGM12} that $g$ is $k$-determined w.r.t. $G$ for all $g\in U_\mu$ (resp. $U_\tau$) and all $k\geq 2\cdot\mu(f)$ (resp. $k\geq 2\cdot\tau(f)$). This means that $k$ is sufficiently large for $f$ w.r.t. $G$. For $f\in\mathfrak{m}\setminus \mathfrak{m}^2$, $f$ is non-singular and the result follows from the implicit function theorem.
\end{proof}

\subsection{Complete unfoldings}\label{sec2.1}
Let $T$ be an {\em (affine) variety over $K$} with structure sheaf $\mathcal{O}_T$ and its algebra of global section $\mathcal{O}(T)=\mathcal{O}_T(T)$. If $F=\sum a_{\alpha}{\bf x}^{\alpha}\in \mathcal{O}(T)[[{\bf x}]], a_{\alpha}\in \mathcal{O}(T)$, then for each $t\in T$, $a_{\alpha}(t)\in K$ denotes the image of $a_{\alpha}$ in $\mathcal{O}_{T,t}/\mathfrak{m}_t=K$, with $\mathfrak{m}_t$ the maximal ideal of the stalk $\mathcal{O}_{T,t}$. Therefore 
$$f_t({\bf x}):=F({\bf x},t)=\sum a_{\alpha}(t){\bf x}^{\alpha}\in K[[{\bf x}]], t\in T,$$
defines a family of power series $f_t$ parametrized by $t\in T$. In the following we often write $f_t({\bf x})$ or $F({\bf x},t)$ instead of $f_t$ or $F$, just to show the variables ${\bf x}$ and the parameter $t\in T$. Moreover if $T=\mathbb A^l=\mathrm{Spec}(K[t_1,\ldots,t_l])$ then $F\in K[{\bf t}][[{\bf x}]]\subset K[[{\bf x,t}]]$ with ${\bf t}:=(t_1,\ldots,t_l)$. Let $f\in \mathfrak{m}$ and let $t_0\in T$. An element $F({\bf x},t)\in \mathcal{O}(T)[[{\bf x}]]$ is called an {\em unfolding} or {\em deformation with trivial section} of $f$ at $t_0\in T$ over $T$ if $f_{t_0}=f$ and $f_t\in \mathfrak{m}$ for all $t\in T$.
 
\begin{definition}\label{def1.4.4.2}{\rm
Let $f\in \mathfrak{m}$ be such that $\mu(f)<\infty$ (resp. $\tau(f)<\infty$), and let $f_t({\bf x})$ be an unfolding of $f$ at $t_0$ over an affine variety $T$. Let $G=\mathcal{R}$ (resp. $G=\mathcal{K}$), let $k$ be sufficiently large for $f$ w.r.t. $G$ and let $\Phi_k$ be the morphism of algebraic varieties from $T$ to the $k$-jet space $J_k$ defined by
$$\Phi_k: T \to J_k,\ t \mapsto  j^k f_t({\bf x}).$$
We define $G\text{-}\mathrm{mod}_F(f):=G_k\text{-}\mathrm{mod}_{\Phi_k}(t_0)$ and call it the {\em $G$-modality} of $f$ w.r.t. the unfolding $F$. Note that $G_k$ acts on on $J_k$ and that $G_k\text{-}\mathrm{mod}_{\Phi_k}(t_0)$ is understood in the sense of Definition \ref{def1.4.2.2}.
}\end{definition}

\begin{proposition}\label{pro1.4.5.0}
For any unfolding $F$ of $f$ at $t_0$, the number $G\text{-}\mathrm{mod}_F(f)$ is independent of the sufficient large integer $k$ for $f$ w.r.t. $G$.
\end{proposition}

\begin{proof}
Let $U_G$ denote the open neighbourhood $U_\mu$ resp. $U_\tau$ of $f$ in $K[[{\bf x}]]$, defined as in the proof of Proposition \ref{pro1.4.3}. It is easy to see that the map
$$\Phi: T\longrightarrow K[[{\bf x}]],\ t \mapsto  f_t({\bf x})$$
is continuous. Then the pre-image $V_G=\Phi^{-1}(U_G)$ is an open neighbourhood of $t_0$. For each $k$ sufficient large for $f$ w.r.t. $G$ we consider the map
$$\varphi_k: V_G \overset{i}{\hookrightarrow } T \overset{\Phi}{\longrightarrow} K[[{\bf x}]] \overset{j^k}{\longrightarrow} J_k.$$
By Corollary \ref{coro1.4.2}, 
$$G\text{-}\mathrm{mod}_{\Phi_k}(t_0)=G\text{-}\mathrm{mod}_{\varphi_k}(t_0)$$
since $\Phi_k=j^k\circ \Phi$. If $k_1, k_2$ are both sufficient large for $f$ w.r.t. $G$, then we can easily check that
$$\varphi_{k_1}^{-1}\big(G_{k_1}\cdot \varphi_{k_1}(t)\big)=\varphi_{k_2}^{-1}\big(G_{k_2}\cdot \varphi_{k_2}(t)\big),\ \forall t\in V_G.$$
Corollary \ref{coro1.4.3} yields that
$$G_{k_1}\text{-}\mathrm{mod}_{\varphi_{k_1}}(t_0)=G_{k_2}\text{-}\mathrm{mod}_{\varphi_{k_2}}(t_0)$$
which proves the proposition.
\end{proof}

\begin{proposition}[Semicontinuity of modality]\label{pro1.4.5.3}
Let $G=\mathcal{R}$ (resp. $G=\mathcal{K}$). Then the $G$-modality is upper semicontinunous, i.e. for all $i\in \mathbb N$, the sets 
$$U_i:=\{f\in \mathfrak{m}\subset K[[{\bf x}]]| G\text{-}\mathrm{mod}(f)\leq i\}$$
are open in $K[[{\bf x}]]$. Consequently, the $G$-modality is upper semicontinunous for unfoldings, i.e. for any unfolding $f_t({\bf x})$ at $t_0$ over $T$ of $f$ with $\mu(f)<\infty$ (resp. $\tau(f)<\infty$) the set
$$\{t\in T| G\text{-}\mathrm{mod}(f_t)\leq G\text{-}\mathrm{mod}(f)\}$$
is open in $T$. 
\end{proposition}

\begin{proof}
Let $f\in U_i$, and let $k\geq 2\mu(f)$ (resp. $k\geq 2\tau(f)$). By Proposition \ref{pro1.4.3} $k$ is sufficiently large for $f$ w.r.t. $G$ and hence $G\text{-}\mathrm{mod}(f)=G_k\text{-}\mathrm{mod}(j^kf)$. Take an open neighbourhood $V$ of $j^kf$ in $J_k$ such that $G_k\text{-}\mathrm{mod}(V)=G_k\text{-}\mathrm{mod}(j^kf)$. Set $U:=(j^k)^{-1}(V)\cap \tilde U$, where
$$\tilde U:=\{g\in K[[{\bf x}]]| \mu(g)\leq \mu(f)\}.$$
By Lemma \ref{pro1.1.1} $\tilde U$ is open and hence $U$ is an open neighbourhood of $f$ in $K[[{\bf x}]]$. We now show that $U\subset U_i$. In fact, for any $g\in U$ one has $k\geq 2\mu(f)\geq 2\mu(g)$ and hence $k$ is sufficiently large for $g$ due to Proposition \ref{pro1.4.3}. Then
$$G\text{-}\mathrm{mod}(g)=G_k\text{-}\mathrm{mod}(j^kg)\leq G_k\text{-}\mathrm{mod}(V)=G_k\text{-}\mathrm{mod}(j^kf)=G\text{-}\mathrm{mod}(f)\leq i.$$
This implies that $U\subset U_i$ and hence $U_i$ is open in $K[[{\bf x}]]$.
\end{proof}

So far we considered families of singularities parametrized by (affine) varieties, in particular by sufficiently high jet spaces. Now we want to use the semiuniversal deformation (with trivial section) of a singularity since its base space has much smaller dimension. However for moduli problems, the formal deformation theory is not sufficient. We have to pass to the \'etale topology and apply Artin's resp. Elkik's algebraization theorems.

Recall that an \'etale neighbourhood of a point $s$ in a variety $S$ consists of a variety $U$ with a point $u\in U$ and an \'etale morphism $\varphi: U\to S$ with $\varphi(u)=s$ (see, \cite[Definition III.5.1]{Mum88}). $\varphi$ is a morphism of pointed varieties, usually denoted by $\varphi: U,u\to S,s$. The connected \'etale  neighbourhoods of $s$ form a filtered system and the direct limit
$$\mathcal{\tilde{O}}_{S,s}:=\lim_{\longrightarrow }\mathcal{O}_{U,u}=\lim_{\longrightarrow}\mathcal{O}(U)$$ 
is called the {\em Henselization} (see \cite{Na53}, \cite{Ra70}, \cite{KPPRM78}) of $\mathcal{O}_{S,s}$. We have $\mathcal{\hat{O}}_{S,s}=\mathcal{\hat{\tilde{O}}}_{S,s}=\mathcal{\hat{O}}_{U,u}$ where $^{\land}$ denotes the completion w.r.t. the maximal ideal. The Henselization of $K[{\bf x}]_{\langle {\bf x}\rangle}$ is the ring $K\langle {\bf x}\rangle$ of algebraic power series in ${\bf x}=(x_1,\ldots,x_n)$. $K[{\bf x}]\subset K\langle {\bf x}\rangle\subset K[[ {\bf x}]]$ and $K\langle {\bf x}\rangle$ may be considered as the union of all $\mathcal{O}(U)\subset K[[ {\bf x}]]$ or $\mathcal{O}_{U,u}\subset K[[ {\bf x}]]$ for $U,u$ an \'etale neighbourhood of $0$ in $\mathbb A^n$. This implies

\begin{remark}\label{rm1.4.0}{\rm 
For each finitely generated subalgebra $A\subset K\langle {\bf x}\rangle$ containing $K[{\bf x}]$ there exist an \'etale neighbourhood $\varphi:V,v_0\to \mathbb A^n,0$ of $0\in \mathbb A^n$ and an inclusion $A\hookrightarrow  \mathcal{O}(V)$, $a({\bf x})\mapsto a_{\varphi}\in \mathcal{O}(V)$. If $a\in K[{\bf x}]$ then $a_{\varphi}(v)=a({\varphi}(v))$ for all $v\in V$, where $a_{\varphi}(v)\in K$ is the image of $a_{\varphi}$ in $\mathcal{O}_{V,v}/\mathfrak{m}_v$.
}\end{remark}

\begin{definition}\label{def1.4.4.3}{\rm
Let $F({\bf x},t)$ be an unfolding of $f$ at $t_0$ over an affine variety $T$. Let $G=\mathcal{R}$ or $G=\mathcal{K}$.
\begin{itemize}
\item[(a)] An unfolding $H({\bf x},s)$ of $f$ at $s_0$ over an affine variety $S$ is called a {\em pullback} or an {\em induced unfolding} of $F$ if there exists a morphism $\psi:S,s_0\to T,t_0$ s.t. $H({\bf x},s)=F({\bf x},\psi(s))$.
\item[(b)] An unfolding $H({\bf x},s)$ of $f$ at $s_0$ over an affine variety $S$ is called an {\em \'etale $G$-pullback} of $F$ if there exist an \'etale neighbourhood $\varphi:W,w_0\to S,s_0$ and a morphism $\psi:W,w_0\to T,t_0$ such that $H({\bf x},\varphi(w))$ is $G$-equivalent to $F({\bf x},\psi(w))$ for all $w\in W$.
\item[(c)] The unfolding $F({\bf x},t)$ is called {\em right } (resp. {\em contact}) {\em complete} (or {\em $G$-complete}) if any unfolding of $f$ is an \'etale $G$-pullback of $F$.
\end{itemize}
}\end{definition}

The following lemma is an immediate consequence of the definition. 
\begin{lemma}\label{coro1.4.5}
Let $H$ be an \'etale $G$-pullback of the unfolding $F$. If $H$ is $G$-complete, then so is $F$. 
\end{lemma}

\begin{proposition}\label{pro1.4.5.2}
Any singularity $f$ with $\mu(f)<\infty$ (resp. $\tau(f)<\infty$) has a right (resp. contact) complete unfolding given by a sufficiently large jet space. More precisely, if $k$ is sufficiently large for $f$ w.r.t. $\mathcal{R}$ (resp. w.r.t. $\mathcal{K}$), then the unfolding of $f$ at $0$ over $J_k=\mathbb A^N$ (with the identification: $\sum_{|\alpha|\leq k} c_{\alpha}{\bf x}^{\alpha}\mapsto c=(c_{\alpha})_{|\alpha|\leq k}, \mathcal{O}(J_k)=K[(c_\alpha)_{|\alpha|\leq k}]$)
$$f_{c}({\bf x})=F({\bf x},c)=f({\bf x})+\sum_{|\alpha|\leq k} c_{\alpha} {\bf x}^{\alpha}\in K[(c_\alpha)_{|\alpha|\leq k}][[{\bf x}]]$$
is right (resp. contact) complete.
\end{proposition}

\begin{proof}
Since $k$ is sufficiently large for $f$ w.r.t. right (resp. contact) equivalence, there exists a neighbourhood $U\subset J_k$ of $j^kf$ such that each $g\in U$ is right (resp. contact) $k$-determined. Let $h_{s}({\bf x}):=H({\bf x},s)$ be an arbitrary unfolding of $f$ at $s_0$ over $S$ and let $W:=\psi^{-1}(U)$ be the pre-image of $U$ by the morphism 
$$\psi:S\longrightarrow \mathbb A^N, s\mapsto j^kh_{s}({\bf x})-j^kf({\bf x}).$$
Then $H({\bf x},s)$ is right (resp. contact) equivalent to $F({\bf x},\psi(s))$ for all $s\in W$ since 
$$j^k H({\bf x},s)=j^k F({\bf x},\psi(s))\in U$$
and hence $H$ is a pullback of $F$. Since every pullback is an \'etale $G$-pullback, this proves the proposition.
\end{proof}

\begin{proposition}\label{pro1.4.5}
Let $f\in K[[{\bf x}]]$ be such that $\mu(f)<\infty$ (resp. $\tau(f)<\infty$) and let $G=\mathcal{R}$ (resp. $G=\mathcal{K}$). Let $F({\bf x},t)$ be an unfolding of $f$ at $t_0$ over $T$.  
\begin{itemize}
\item[(i)] If the unfolding $H({\bf x},s)$ of $f$ at $s_0$ over $S$ is an \'etale $G$-pullback of $F$, then $G\text{-}\mathrm{mod}_{H}(f)\leq G\text{-}\mathrm{mod}_{F}(f).$
\item[(ii)] We always have $G\text{-}\mathrm{mod}(f)\geq G\text{-}\mathrm{mod}_{F}(f)$. Equality holds if $F({\bf x},t)$ is $G$-complete. 
\end{itemize}
\end{proposition}

\begin{proof}
(i) Since $H$ is an \'etale $G$-pullback of $F$, there exist an \'etale neighbourhood $\varphi:W,w_0\to S,s_0$ and a morphism $\psi:W,w_0\to T,t_0$ such that $G\cdot H({\bf x},\varphi(w))=G\cdot F({\bf x},\psi(w))$ for all $w\in W$.

Let $k$ be sufficiently large for $f$ and let $\Phi_k$ and $\Psi_k$ be the morphisms defined by
$$\Phi_k: T\to J_k, t\mapsto j^kf_{t}({\bf x})\text{ and }\Psi_k: S\to J_k, s\mapsto j^kh_{s}({\bf x}).$$
Then $\Phi^{-1}(G\cdot \Phi({w}))=\Psi^{-1}(G\cdot \Psi({w}))$ for all $w\in W$ with $\Phi:=\Phi_k\circ \psi$ and $\Psi:=\Psi_k\circ \varphi$. Combining Corollary \ref{coro1.4.2} and \ref{coro1.4.3} we obtain
$$G\text{-}\mathrm{mod}_{F}(f)=G_k\text{-}\mathrm{mod}_{\Phi_k}(t_0)\geq G_k\text{-}\mathrm{mod}_{\Phi}(w_0)=G_k\text{-}\mathrm{mod}_{\Psi}(w_0)=G_k\text{-}\mathrm{mod}_{\Psi_k}(s_0)=G\text{-}\mathrm{mod}_{H}(f).$$

(ii) Since $k$ is sufficiently large for $f$ w.r.t. $G$,
$$G\text{-}\mathrm{mod}(f)=G_k\text{-}\mathrm{mod}(j^k f)\text{ and } G\text{-}\mathrm{mod}_{F}(f)=G_k\text{-}\mathrm{mod}_{\Phi_k}(t_0).$$
It follows from Corollary \ref{coro1.4.1} that $G_k\text{-}\mathrm{mod}(j^k f)\geq G_k\text{-}\mathrm{mod}_{\Phi_k}(t_0).$ Hence
$$G\text{-}\mathrm{mod}(f)\geq G\text{-}\mathrm{mod}_{F}(f).$$
We identify $J_k$ with $\mathbb A^N$ via the map $\sum_{|\alpha|\leq k} c_{\alpha}{\bf x}^{\alpha}\mapsto c:=(c_{\alpha})_{|\alpha|\leq k}$ and consider the unfolding 
$$h_c({\bf x}):=H({\bf x},c):=f({\bf x})+\sum_{|\alpha|\leq k} c_{\alpha} {\bf x}^{\alpha}$$ 
of $f$ at $0$ over $\mathbb A^N$. Since the translation
\begin{displaymath}
\begin{array}{rccccc}
\psi: \mathbb A^N&\longrightarrow & J_k&=&\mathbb A^N\\
c &\mapsto & j^k h_{c}({\bf x})&=& j^kf+c
\end{array}
\end{displaymath}
is an isomorphism, it follows from Corollary \ref{coro1.4.1} that
$$G_k\text{-}\mathrm{mod}(j^k f) =G_k\text{-}\mathrm{mod}_{\psi}(0)$$
and hence $G\text{-}\mathrm{mod}(f)=G\text{-}\mathrm{mod}_{H}(f)$.

Now, if the unfolding $F$ is $G$-complete, then $H$ is an \'etale $G$-pullback of $F$. By (i), $G\text{-}\mathrm{mod}_H(f) \leq G\text{-}\mathrm{mod}_{F}(f)$ and hence $G\text{-}\mathrm{mod}(f)=G\text{-}\mathrm{mod}_F(f)$.
\end{proof}

\begin{example}\label{rm1.4.1}{\rm
(a) The unfolding $F(x,t)= x^{p+1}+t_1 x+\ldots+t_{p} x^{p}$ of the univariate polynomial $f=x^{p+1}$ over $T=\mathbb A^{p}$ is right complete. Here $(t_1,\ldots,t_p)$ are the coordinates of $t\in T$.

Indeed, for any unfolding $H(x,s)$ of $f$ at $s_0$ over some variety $S$ we write
$$H(x,s)=a_1(s)x+\ldots+a_p(s)x^{p}+a_{p+1}(s)x^{p+1}+\ldots$$
with $a_i(s)\in \mathcal{O}(S)$. Then $a_i(s_0)=0\ \forall i\leq p$ and $a_{p+1}(s_0)\neq 0$. Consider the morphism $\varphi:S\to T,\ s\mapsto \big(a_{1}(s),\ldots, a_{p}(s)\big)$ and the open neighbourhood $W:=\{s \in S\ |\ a_{p+1}(s)\neq 0\}$ of $s_0$ in $S$. Then it follows from \cite[Thm. 2.1]{BGM12} that 
$$F(x,\varphi(s))\sim_r H(x,s), \text{ for each } s\in W.$$ 
Note that $\{x,\ldots,x^p\}$ is a basis of $\mathfrak{m}/\mathfrak{m}j(f)$ and that $F$ is a semiuniversal deformation of $f$ with trivial section by Proposition \ref{pro1.4.4.1}. 

(b) The right semi-universal deformation (without section) of $f=x^{p+1}\in K[[x]]$ with $\mathrm{char}(K)=p>2$ is given by $H(x,t)=x^{p+1}+t_1x+\ldots+t_{p-1} x^{p-1}$. This unfolding of $f$ over $\mathbb A^{p-1}$ is not right complete. In fact, it is not difficult to see that $H(x,t)$ is equivalent to one of $\{x,\ldots, x^{p-1}, x^{p+1}\}$ for $t\in \mathbb A^{p}$. Corollary \ref{coro1.4.6} yields that $\mathcal{R}\text{-}\mathrm{mod}_H(f)=0$, while $\mathcal{R}\text{-}\mathrm{mod}(f)>0$ by Theorem \ref{thm1.5.0} and hence $H$ is not right complete due to Proposition \ref{pro1.4.5.3}.

To see this directly, consider the family $x^{p+1}+s x^{p}$ in characteristic $p>0$ over $\mathbb A^1$ which, as an unfolding with trivial section, cannot be induced by a morphism $\varphi:\mathbb A^1\to \mathbb A^{p-1}$: Since $H(x,\varphi(s))$ has multiplicity $\neq p$ in $K[[x]]$ for all $s\neq 0$, it cannot be right equivalent to $x^{p+1}+s x^{p}$ which has multiplicity $p$ for $s\neq 0$.

This is of course not a contradiction to $F$ being versal as deformation without section, which means that the family $x^{p+1}+s x^{p}\in K[[x,s]]$ can be induced from $H$ (up to isomorphism in $K[[x,s]]$ over $K[[s]]$) by a morphism $\varphi\colon K[[t_1,\ldots,t_{p-1}]] \to K[[s]]$. In fact, define $\varphi$ by $t_1\mapsto -s^p, t_i\mapsto 0$ for $i>1$, then, if $\mathrm{char}(K)=p$,
$$H(x,\varphi(s))=-s^px+x^{p+1}=(x-s)^p x\sim_r x^{p+1}+s x^{p},$$
via the isomorphism $\Phi: K[[x,s]]\to K[[x,s]]$ over $K[[s]]$, given by $x\mapsto x-s, s\mapsto s$. However, $\Phi$ does not respect the trivial section.
 
 If $\mathrm{char}(K)\neq p$, we can use the Tschirnhaus transformation $x\mapsto x-\dfrac{s}{p}$ to eliminate $sx^p$ from $x^{p+1}+s x^{p}$ and to show that $x^{p+1}+s x^{p}$ can be induced from $H$. 
}\end{example} 

The following proposition is stronger than Proposition \ref{pro1.4.5.2}  because it reduces the number of parameters of a $G$-complete unfolding considerably. For the proof we need the nested Artin approximation theorem.

\begin{proposition}\label{pro1.4.4.1}
Let $f\in \mathfrak{m}\subset K[[{\bf x}]]$ be such that $\mu(f)<\infty$ (resp. $\tau(f)<\infty$). Let $g_1,\ldots, g_{l}\in K[{\bf x}]$ be a system of $K$-- generators of $\mathfrak{m}/\mathfrak{m}\cdot j(f)$ (resp. of $\mathfrak{m}/\langle f\rangle +\mathfrak{m}\cdot j(f)$), with $j(f)$ the jacobian ideal of $f$. Then the unfolding (with trivial section) of $f$ over $\mathbb A^l$,
$$F({\bf x},t)=f({\bf x})+\sum_{i=1}^l t_i g_i({\bf x}) \in K[{\bf t}][[{\bf x}]]$$
is (an algebraic representative of) a formally versal deformation of $f$ with trivial section with respect to right (resp. contact) equivalence, which is semi-universal if the system $\{g_i\}$ is a basis. This unfolding is right (resp. contact) complete.
\end{proposition}

\begin{proof}
We first show that if $\{g_i\}$ is a basis, then $F$ represents a formally semiuniversal deformation of $f$ with trivial section with respect to right (resp. contact) equivalence. Indeed we may consider $F$ as a deformation of $f$ over $K[[{\bf t}]]:=K[[t_1,\ldots,t_l]]$. It is shown in \cite[Prop. 2.7]{BGM12} that the tangent space to the base space of the semiuniversal deformation with trivial section is $\mathfrak{m}/\mathfrak{m}\cdot j(f)$ (resp. $\mathfrak{m}/\langle f\rangle +\mathfrak{m}\cdot j(f)$). The proof of the existence of a semiuniversal deformation in \cite{KaS72} or \cite[Thm. II.1.16]{GLS06} can be easily adapted to deformations with section, showing the versality of $F$ and hence proving the first claim. 

Let $G=\mathcal{R}$ (resp. $G=\mathcal{K}$) and let $k$ be sufficiently large for $f$ w.r.t. $G$. Let $j^k$ denote the projections $K[{\bf t}][[{\bf x}]]\to K[{\bf t}][[{\bf x}]]/\langle {\bf x}\rangle^{k+1}$ and $K[[{\bf x}]]\to K[[{\bf x}]]/\langle {\bf x}\rangle^{k+1}$. Then we may replace $F$ resp. $f$ by $j^kF$ resp. $j^kf$ and assume that $F=j^kF\in K[{\bf x,t}]$ and $f=j^kf\in K[{\bf x}]$ by the following facts:

\begin{itemize}
\item[(a)] $\{j^kg_1,\ldots, j^kg_{l}\}$ is a system of generators of $\mathfrak{m}/\mathfrak{m}\cdot j(j^kf)$ (resp. of $\mathfrak{m}/\langle f\rangle +\mathfrak{m}\cdot j(j^kf)$) with $j(j^kf)$ the jacobian ideal of $j^kf$.
\item[(b)] $F$ is $G$-complete if and only if $j^kF({\bf x},t)$ is a $G$-complete unfolding of $j^k f({\bf x})$ at $t_0$ over $T$ (see \cite[Prop. 3.4.12]{Ng13}).
\end{itemize}
Consider the complete unfolding of $f$ over $\mathbb A^N=J_k$ at $0$
$$h_c({\bf x}):=H({\bf x},c)=f({\bf x})+\sum_{|\alpha|\leq k} c_{\alpha} {\bf x}^{\alpha}$$ 
 as in Proposition \ref{pro1.4.5.2}. By Lemma \ref{coro1.4.5}, the proof is completed by showing that $H$ is an \'etale $G$-pullback of $F$. 

In fact, consider $H$ as a deformation with trivial section of $f$ over $K[[{\bf c}]]:=K[[(c_{\alpha})_{|\alpha|\leq k}]]$. Although $F$ and $H$ are polynomials, the versality of $F$ implies only that $H$ can be induced formally from $F$. This means, there exist tuples of formal power series $\hat\Phi=(\hat\Phi_1,\ldots,\hat\Phi_n)\in \big (\langle {\bf x}\rangle K[[{\bf x,c}]]\big )^n$, $\hat\varphi=(\hat\varphi_1,\ldots,\hat\varphi_l)\in \big (\langle {\bf c}\rangle K[[{\bf c}]]\big )^l$ and a unit $\hat u({\bf x,c})\in K[[{\bf x,c}]]$ (with $\hat u=1$ for $G=\mathcal{R}$) with

\begin{equation}
\hat u({0,0})\neq 0 \text{ and }\Big(\det\big( \frac{\partial \hat\Phi_i}{\partial x_j}\big)\Big)|_{(0,0)}\neq 0,
\end{equation}

such that
$$H\big(\hat\Phi({\bf x,c}\big),{\bf c})=\hat u({\bf x,c})\cdot F\big({\bf x},\hat\varphi({\bf c})\big).$$

Let ${\bf y}=(y_1,\ldots,y_{n+l+1})$ be new indeterminates (omitting $y_{n+l+1}$ if $G=\mathcal{R}$) and let
$$P({\bf x,c,y})=H(y_1,\ldots,y_n,{\bf c})-y_{n+l+1}\cdot F({\bf x},y_{n+1},\ldots,y_{n+l})\in K[{\bf x,c,y}].$$

The formal versality of $F$ implies that $P=0$ has a formal solution $\hat{\bf y}=(\hat y_1,\ldots, \hat y_{n+l+1})$ with 
\begin{eqnarray*}
\hat{y}_i&:=&\hat\Phi_i\in K[[{\bf x,c}]], 1\leq i\leq n,\\
\hat{y}_{n+j}&:=&\hat\varphi_j \in K[[{\bf c}]], 1\leq j\leq l,\\
\hat{y}_{n+l+1}&:=&\hat u\in K[[{\bf x,c}]].
\end{eqnarray*}

By the nested Artin Approximation Theorem (\cite[Theorem 1.4]{Po86}), there exists an algebraic solution $\tilde{\bf y}=(\tilde{y}_1,\ldots, \tilde{y}_{n+l+1})$ of $P=0$ such that
\begin{eqnarray*}
\tilde{y}_i&\in& K\langle{\bf x,c}\rangle, 1\leq i\leq n,\\
\tilde{y}_{n+j}&\in & K\langle{\bf c}\rangle, 1\leq j\leq l,\\
\tilde{y}_{n+l+1}&\in & K\langle{\bf x,c}\rangle
\end{eqnarray*}
and 
\begin{equation}
\tilde{\bf y}({\bf c,x})-\hat{\bf y}({\bf c,x})\in \langle{\bf c,x}\rangle^2.
\end{equation}

Passing to the $k$-jet spaces by the projection $j^k:K[[{\bf x,c}]]\to K[[{\bf x,c}]]/\langle{\bf x}\rangle^{k+1}$ we get 
$$j^k\big(P({\bf x,c},j^k(\tilde{\bf y}))\big)=j^k\big(H(j^k(\tilde{y}_1),\ldots,j^k(\tilde{y}_n),{\bf c})\big)-j^k(\tilde{y}_{n+l+1})\cdot j^k\big(F({\bf x},j^k(\tilde{y}_{n+1}),\ldots,j^k(\tilde{y}_{n+l}))\big)=0.$$
Let $A\subset K[[{\bf c}]]$ be the subalgebra generated by $c_{\alpha}, |\alpha|\leq k$, and all the coefficients of ${\bf x}^{\alpha}, |\alpha|\leq k$, which appear in all $j^k(\tilde{y}_i)({\bf x,c}), i=1,\ldots,n+l+1$. Then  $A\subset K\langle{\bf c}\rangle$ since $K\langle{\bf c,x}\rangle\subset K\langle{\bf c}\rangle[[{\bf x}]]$ (cf. \cite[Lemma 3.4.9]{Ng13}). 
It follows from Remark \ref{rm1.4.0} that there exists an \'etale neighbourhood $\varphi: V,v_0 \to \mathbb A^N,0$ and an inclusion $A\hookrightarrow \mathcal{O}(V), a\mapsto a_{\varphi}$ such that if $a\in K[{\bf c}]$ the $a_{\varphi}(v)=a(\varphi(v))$ for all $v\in V$. This implies that, denoting by $\iota:A[{\bf x}]\hookrightarrow \mathcal{O}(V)[{\bf x}]$ the induced inclusion, 
$$j^k\big(P({\bf x},\varphi(v), \Phi({\bf x},v)\big)=0\text{ in } \mathcal{O}(V)[{\bf x}],$$
where $\Phi({\bf x},v)=(\Phi_1({\bf x},v),\ldots,\Phi_n({\bf x},v))$ with $\Phi_i(({\bf x},v)))=\iota (j^k(\tilde{y}_i)({\bf x,c}))$. Combining (2.1), (2.2) we obtain that there exists an open neighbourhood $W'\subset V$ of $v_0$ such that for all $v\in W'$,
$$\Phi({\bf x},v) \in Aut(K[[{\bf x}]])\text{ and } u({\bf x},v):=\iota(j^k \tilde{y}_{n+l+1}({\bf x,c})) \in K[[{\bf x}]]^{*}.
$$
This implies $j^kH(\Phi({\bf x},v),\varphi(v))$ is $G$-equivalent to $F({\bf x},\psi(v))$, where
$$\psi:V,v_0\to \mathbb A^l,\ \psi(v)=(\psi_1(v),\ldots,\psi_l(v))$$ 
with $\psi_j(v)=\iota(j^k(\tilde{y}_{n+j})({\bf c}))$. Moreover since $k$ is sufficiently large for $f$ w.r.t. $G$, there exists an open neighbourhood $W\subset W'$ of $v_0$ such that 
$$H(\Phi({\bf x},v),\varphi(v))\sim_G j^kH(\Phi({\bf x},v),\varphi(v)),\ \forall v\in W.$$
Hence
$$ H(\Phi({\bf x},v),\varphi(v))\sim_G j^kH(\Phi({\bf x},v),\varphi(v))\sim_G F({\bf x},\psi({v})), \forall u\in W,$$
which proves the claim.
\end{proof}

The following proposition will be used to show how the modality of $f$ is related to the number of parameters in families of normal forms. To make this precise we introduce $G$-modular families. By a {\em $G$-modular family} over a variety $S$ we mean a family $h_s({\bf x})=H({\bf x},s)\in \mathcal O(S)[[{\bf x}]]$ such that for each $s\in S$ there are only finitely many $s'\in S$ such that $h_s$ is $G$-equivalent to $h_{s'}$. 

\begin{proposition}\label{pro}
Let $f\in K[[{\bf x}]]$ be such that $\mu(f)<\infty$ (resp. $\tau(f)<\infty$) and let $G=\mathcal{R}$ (resp. $G=\mathcal{K}$). Let $f_t({\bf x})=F({\bf x},t)$ be any unfolding of $f$ at $t_0$ over $T$. Assume that there exist an open neighbourhood $W\subset T$ of $t_0$ and families $h^{(i)}_{s_i}({\bf x})=H^{(i)}({\bf x},s_i)\in \mathcal O(S_i)[[{\bf x}]]$ over affine varieties $S_i, i=1,\ldots,q$, such that for all $t\in W$ there exists $s_i\in S_i$ such that $f_t\sim_G h^{(i)}_{s_i}$.
\begin{itemize}
\item[(i)] We have
$$G\text{-}\mathrm{mod}_F(f)\leq \max_{i=1,\ldots,q}\{\dim S_i\}.$$
\item[(ii)] Assume moreover that each family $h^{(i)}_{s_i}({\bf x}), i=1,\ldots,q$, is $G$-modular and that for each open neighbourhood $V$ of $t_0$ in $W$ and for all $s_i\in S_i$ there exists a $t\in V$ such that $f_t\sim_G h^{(i)}_{s_i}$. Then
$$G\text{-}\mathrm{mod}_F(f)=\max_{i=1,\ldots,q} \{\dim S_i\}.$$
\end{itemize}
\end{proposition}

\begin{proof}
(i) Let $k$ be sufficiently large for $f$ w.r.t. $G$. Considering the morphisms 
$$\Phi_k: T\to J_k,\ t \mapsto j^k f_t({\bf x})\text{ and }\Phi^{(i)}_k: S_i\to J_k,\ s\mapsto j^k h^{(i)}_s({\bf x}),  i=1,\ldots,q,$$
and applying Corollary \ref{coro1.4.4}(i) we obtain that
$$G_k\text{-}\mathrm{mod}_{\Phi_k}(W)\leq \max_{i=1,\ldots,q}\{G_k\text{-}\mathrm{mod}_{\Phi^{(i)}_k}(S_i)\}.$$
Hence
$$G\text{-}\mathrm{mod}_F(f)=G_k\text{-}\mathrm{mod}_{\Phi_k}(0)\leq G_k\text{-}\mathrm{mod}_{\Phi_k}(W)\leq \max_{i=1,\ldots,q}\{G_k\text{-}\mathrm{mod}_{\Phi^{(i)}_k}(S_i)\}\leq  \max_{i=1,\ldots,q}\{\dim S_i\}.$$

(ii) Let $V$ be an open neighbourhood of $0$ such that $G\text{-}\mathrm{mod}_F(f)=G_k\text{-}\mathrm{mod}_{\Phi_k}(0) = G_k\text{-}\mathrm{mod}_{\Phi_k}(V)$. It follows from Corollary \ref{coro1.4.4}(ii) that
$$G_k\text{-}\mathrm{mod}_{\Phi_k}(V)=\max_{i=1,\ldots,q}\{G_k\text{-}\mathrm{mod}_{\Phi^{(i)}_k}(S_i)\}.$$
Moreover since the $h^{(i)}_{s_i}$ are modular, we can see, with the notations of Definition \ref{def1.4.2.1}, that for each $s_i\in S_i$ the set $V_{G_k,\Phi^{(i)}_k}(s_i)$ is finite and hence $G_k\text{-}\mathrm{mod}_{\Phi^{(i)}_k}(S_i)=G_k\text{-}\mathrm{par}_{\Phi^{(i)}_k}(S_i)=\dim S_i$. This completes the proof.
\end{proof}

Combining Propositions \ref{pro1.4.5} and \ref{pro} we get

\begin{corollary}\label{}
Let $f\in K[[{\bf x}]]$ be such that $\mu(f)<\infty$ (resp. $\tau(f)<\infty$) and let $G=\mathcal{R}$ (resp. $G=\mathcal{K}$). Let $f_t({\bf x})=F({\bf x},t)$ be a $G$-complete unfolding of $f$ at $t_0$ over $T$ (e.g. $F$ an algebraic representative of a $G$-versal deformation with trivial section of $f$ as in Proposition \ref{pro1.4.4.1}). Assume that there are finitely many $G$-modular families $h^{(i)}_{s_i}({\bf x})=H^{(i)}({\bf x},s_i)\in \mathcal O(S_i)[[{\bf x}]]$ and a neighbourhood $W\subset T$ of $t_0$ such that for each open neighbourhood $V\subset W$ of $t_0$ we have
$$\bigcup_{t\in V} G\cdot f_{t}=\bigcup_{i}\bigcup_{s_i\in S_i} G\cdot h^{(i)}_{s_i}.$$
Then $G\text{-}\mathrm{mod}(f)=\max \{\dim S_i\}$.
\end{corollary}

\begin{proof}
By Proposition \ref{pro1.4.4.1} an algebraic representative of any $G$-versal deformation (with section) of $f$ is $G$-complete. By Proposition \ref{pro1.4.5} $G\text{-}\mathrm{mod}(f)=G\text{-}\mathrm{mod}_F(f)$. The rest follows from Proposition \ref{pro}(ii).
\end{proof}

Note that in his classification of right simple, unimodal and bimodal singularities Arnol'd constructed parametrized {\em normal forms}, being actually $\mathcal R$-modular families of dimension 0, 1 and 2.

The above corollary makes precise (and proves) the statement by Wall \cite{Wal83} for complex analytic singularities saying: ``if the set of germs $f_t$ ($t$ small) at points $x$ near $0$ falls into finitely many families of right (resp. contact) equivalence classes, each depending on $r$ parameters (at most) then $f$ is right (resp. contact) $r$-modal (at most).''

\begin{corollary}\label{coro1.4.6}
Let $f\in K[[{\bf x}]]$ be such that $\mu(f)<\infty$ resp. $\tau(f)<\infty$. $f$ is $G$-simple iff it is of {\em finite $G$-unfolding type}, i.e. there exists a finite set $\mathcal{F}$ of $G$-classes of singularities satisfying: for any (or, equivalently, for one $G$-complete) unfolding $F({\bf x},t)$ of $f$ at $t_0$ over an affine variety $T$, there exists a Zariski open neighbourhood $V$ of $t_0\in  T$, such that the set of $G$-classes of singularities of $F({\bf x},t)$, $t \in V$, belongs to the set $\mathcal{F}$.
\end{corollary}

\begin{proposition}\label{pro1.4.4}
Let $f\in K[[{\bf x}]]=K[[x_1,\ldots, x_n]]$ be such that $\mu(f)<\infty$ (resp. $\tau(f)<\infty$) and let $G=\mathcal{R}$ (resp. $G=\mathcal{K}$). Let $m=mt(f)$ be the multiplicity of $f$. Then 
$$G\text{-}\mathrm{mod}(f) \geq \binom{n+m-1}{m}-n^2.$$
 In particular,
\begin{itemize}
\item[(i)] if $n= 2$, then $G\text{-}\mathrm{mod}(f) \geq m-3$;
\item[(ii)] if $n\geq 3$ and $m\geq 3$, then $G\text{-}\mathrm{mod}(f) \geq \dfrac{m^2+3m-16}{2}$.
\end{itemize}
\end{proposition}

\begin{proof}
Let $k$ be sufficiently large for $f$ w.r.t. $G$. Then $k\geq m$ and $G\text{-}\mathrm{mod}(f)=G_k\text{-}\mathrm{mod}(j^k f)$. Put $X:=\mathfrak{m}^{m}/\mathfrak{m}^{k+1}\subset J_k$. It follows from Proposition \ref{pro1.4.1}(iii) that 
$$G\text{-}\mathrm{mod}(f) \text{ in } J_k\geq G\text{-}\mathrm{mod}(f)\text{ in } X.$$
Let the linear group $G':=GL(n,K)$ act on $X':=\mathfrak{m}^{m}/\mathfrak{m}^{m+1}$ by $G'\times X'\to X', \ (A,g({\bf x})) \mapsto  g(A{\bf x})$. Consider the projection $p: X\to X'$. It is easy to see that $p$ is open and $G\cdot g \subset p^{-1}\big(G'\cdot p(g)\big)$ for all $g\in X$. Then Proposition \ref{pro1.4.1}(iv) yields 
$$G\text{-}\mathrm{mod}(g) \geq G'\text{-}\mathrm{mod}(p(g)),\ \forall g\in X.$$
In order to prove the proposition, it is sufficient to show that 
$$G'\text{-}\mathrm{mod}(p(g)) \geq \binom{n+m-1}{m}-n^2, \forall g\in X.$$
Indeed, it is easy to see that 
$$\dim X'= \binom{n+m-1}{m}\text{ and }\dim GL(n,K)=n^2.$$
Hence, by Proposition \ref{pro1.4.1}(i),
\begin{eqnarray*}
GL(n,K)\text{-}\mathrm{mod}(p(g))&\geq & \dim X'-\dim GL(n,K)\\
&=&\binom{n+m-1}{m}-n^2,
\end{eqnarray*}
which completes the proof.

(i) and (ii) follow from explicit calculations.
\end{proof}

\section{Classification of right simple singularities}
In this section we classify the right simple singularities $f\in \mathfrak{m}\subset K[[x_1,\ldots,x_n]]$ for $K$ an algebraically closed field of characteristic $p>0$. The classification of contact simple singularities was done in \cite{GK90}. In contrast to $\mathrm{char}(K)=0$, where the classification of right simple and contact simple singularities coincides, the classification is very different in positive characteristic. For example, for every $p>0$, there are only finitely many classes of right simple singularities and for $p=2$ only the $A_1$-singularity in an even number of variables is right simple. The classification of right simple singularities is summarized in Theorems \ref{thm1.5.0}, \ref{thm1.5.1} and \ref{thm1.5.2}.
Note that $f\in\mathfrak{m}\setminus\mathfrak{m}^2\Leftrightarrow \mu(f)=0$ and then $f\sim_r x_1$ (hence right simple) by the implicit function theorem. We may therefore assume in the following that $f\in\mathfrak{m}^2$. 

\begin{theorem}\label{thm1.5.0}
Let $\mathrm{char } K=p > 0$.  Let $f\in\mathfrak{m}^2\subset K[[x]]$ be a univariate power series such that its Milnor number $\mu:=\mu(f)$ is finite. Then
$$\mathcal{R}\text{-}\mathrm{mod}(f)=[\mu/p]$$
is the integer part of $\mu/p$. In particular, $f$ is right simple if and only if $\mu<p$, and then $f\sim_r x^{\mu+1}$.
\end{theorem}

\begin{theorem}\label{thm1.5.1}
Let $p=\mathrm{char}(K)>2$. 
\begin{itemize}
\item[(i)] A plane curve singularity $f\in \mathfrak{m}^2\subset K[[x,y]]$ is right simple if and only if it is right equivalent to one of the following normal forms
$$$$
\centerline{\begin{tabular}[20pt]{|c|l l|}
\hline 
Name& \multicolumn{2}{l|}{Normal form}\\
\hline
$\mathrm{A}_k$&$x^2+y^{k+1}$&$1\leq k\leq p-2$\\
\hline
$\mathrm{D}_k$&$x^2y+y^{k-1}$& $4\leq k< p$\\
\hline
$\mathrm{E}_6$& $x^3+y^4$&$3<p$\\
\hline
$\mathrm{E}_7$&  $x^3+xy^3$&$3<p$\\
\hline
$\mathrm{E}_8$& $x^3+y^5$&$5<p$\\
\hline
\multicolumn{3}{c}{ }\\
\multicolumn{3}{c}{\rm{Table \ref{thm1.5.1} (i)}}
\end{tabular}}\\

\item[(ii)] A hypersurface singularity $f\in\mathfrak{m}^2\subset K[[x_1,\ldots,x_n]],n\geq 3,$ is right simple if and only if it is right equivalent to one of the following normal forms
$$$$
\centerline{\begin{tabular}[20pt]{|l| l|}
\hline 
\multicolumn{2}{|l|}{Normal form}\\
\hline 
$g(x_1,x_2)+x_3^2+\ldots+x_n^2$& $g$ is one of the singularities in Table \ref{thm1.5.1} (i)\\
\hline
\multicolumn{2}{c}{ }\\
\multicolumn{2}{c}{\rm{Table \ref{thm1.5.1} (ii)}}
\end{tabular}}
\end{itemize}
\end{theorem}

\begin{theorem}\label{thm1.5.2}
Let $p=\mathrm{char}(K)=2$. A hypersurface singularity $f\in\mathfrak{m}^2\subset K[[x_1,\ldots,x_n]]$ with $n\geq 2$, is right simple if and only if $n$ is even and if it is right equivalent to 
$$\mathrm{A}_1 : x_1x_2+x_3x_4+\ldots+x_{n-1}x_{n}.$$
\end{theorem}

The following interesting corollary follows immediately from the classification of right simple singularities.

\begin{corollary}\label{coro1.5.1}
For any $p=\mathrm{char}(K)>0$ there are only finitely many right simple singularities $f\in\mathfrak{m}^2\subset K[[x_1,\ldots,x_n]]$ (up to right equivalence). For $p=2$, either no or exactly one right simple singularity exists. 
\end{corollary}

The corollary also shows that if $f_k,k\geq 1$, is any sequences of simple singularities in positive characteristic then the sequence of Milnor numbers $\mu(f_k),k\geq 1$, is bounded. Note that this is wrong in characteristic zero since the $A_k,D_k,k\geq 1$, with Milnor number $k$, are all simple. We like to pose the following conjecture:

\begin{conjecture}\label{conj}
Let $f_k\in K[[x_1,\ldots,x_n]],\mathrm{char}(K)=p>0$, be a sequence of isolated singularities with Milnor number going to infinity if $k\to \infty$. Then the right modality of $f_k$ goes to infinity. 
\end{conjecture}

\subsection{Univariate singularities in positive characteristic}\label{sec3.1}
Let $K[[x]]$ be the ring of univariate formal power series. It is obvious that $G\text{-}\mathrm{mod}(f)=0$ for all $f\in K[[x]]$ if either $p:=\mathrm{char}(K)=0$ or $G=\mathcal{K}$. For the complete classification of univariate singularities with any modality and for the proof of Theorem \ref{thm1.5.0} we refer to \cite[Thm. 4.2.8]{Ng13}, \cite[Thm. 3.1 ]{Ng12}. Here we prove only the second part, i.e. $f$ is right simple iff $\mu<p$ and then $f\sim_r x^{\mu+1}$. The ``if''-statement follows easily from the upper semi-continuity of the Milnor number and the following fact: If $p\nmid mt(f)$ (in particular, if $\mu<p$) then $f\sim_r x^{mt(f)}$ (cf. \cite[Cor. 1]{BGM12}). 

It suffices to show that if $\mu\geq p$ then $\mathcal{R}\text{-}\mathrm{mod}(f)\geq 1$. Indeed, since $\mu\geq p$, $m:=mt(f)\geq p$. 

If $m=p$ then we may assume that $f=x^p+a_{p+1}x^{p+1}+\ldots\in K[[x]]$. Consider the unfolding
$$f_t=f+tx^{p+1}=x^p+(t+a_{p+1})x^{p+1}+\ldots.$$
of $f$. We show that  $f_t\sim_r f_{t'}$ implies $t=t'$. If $\varphi(x)=u_1x+u_2x^2+\ldots$ is in $\mathcal{R}$ then $u_1\neq 0$ and 
$$\varphi(f_t)=u_1^px^p+u_2^px^{2p}+\ldots+ (t+a_{p+1}) u_1^{p+1}x^{p+1}+\ldots.$$
If $\varphi(f_t)=f_{t'}$ then $u_1^p=1$, hence $u_1=1$, and $t=t'$. This implies that $\mathcal{R}\text{-}\mathrm{mod}(f)\geq 1$ by Corollary \ref{coro1.4.6}.

Now, assume that $m>p$ and consider the unfolding $g_t:=G(x,t):=f+t\cdot x^{p}$ of $f$ at $0$ over $\mathbb A^1$. By Proposition \ref{pro1.4.5.3}, there exists an open neighbourhood $V$ of $0$ in $\mathbb A^1$ such that 
$\mathcal{R}\text{-}\mathrm{mod}(g_{t})\leq \mathcal{R}\text{-}\mathrm{mod}(f)$ for all $t\in V$. Take a $t_0\in V\setminus \{0\}$, then the above case with $mt=p$ yields that $\mathcal{R}\text{-}\mathrm{mod}(g_{t_0})\geq 1$ since $mt(g_{t_0})=p$, and hence
$$\mathcal{R}\text{-}\mathrm{mod}(f)\geq \mathcal{R}\text{-}\mathrm{mod}(g_{t_0})\geq 1.$$

\subsection{Right simple plane curve singularities in characteristic $>2$}
Here and in the next section let $f\in K[[x,y]]$, $mt(f)$ its multiplicity, and $\mu=\mu(f)$ its Milnor number, which we assume to be finite. Let $p=\mathrm{char}(K)$.

\begin{proposition}\label{pro1.5.1}
Let $mt(f)=2$ and $p>2$.
\begin{itemize}
\item[(i)] If $\mu<p-1$, then $f\sim_r A_\mu$ and $f$ is right simple.
\item[(ii)] If $\mu\geq p-1$, then $f$ is not right simple.
\end{itemize} 
\end{proposition}

\begin{proof}
Since $mt(f)=2$ and $p>2$, it follows from the right splitting lemma (Lemma \ref{lm1.5.3.0}) that $f$ is right equivalent to $x^2+g(y)$ (with $g(y)=y^2$ if $\mathrm{crk}(f)=0$ (case $A_1$) and $mt(g)\geq 3$ if $\mathrm{crk}(f)=1$). Here $\mathrm{crk}(f)$ denotes the corank of the Hessian of $f$, see Section \ref{sec3.3}.

(i) If $\mu<p-1$ then $mt(g)<p$. By Theorem \ref{thm1.5.0}, $g\sim_r y^{mt(g)}$ and hence $f\sim_r A_\mu$. Moreover, Theorem \ref{thm1.5.0} yields that $g$ is right simple and so is $f$ by Lemma \ref{lm1.5.4}(iii). 

(ii) If $\mu\geq p-1$, then $mt(g)\geq p$. Combining Theorem \ref{thm1.5.0} and  Lemma \ref{lm1.5.4}(iii) we get that $f$ is not right simple.
\end{proof}

\begin{proposition}\label{pro1.5.2}
Let $p>3$, let $mt(f)=3$ and  and $f_3$ be the {\em tangent cone} (i.e. the homogeneous component of degree $3$) of $f$. Let $r(f_3)$ be the number of linear factors of $f_3$.
\begin{itemize}
\item[(i)] If $r(f_3)\geq 2$ then $f\sim_r x^2y+g(y)$ with $mt(g)=\mu-1\geq 3$. If additionally $4 \leq \mu < p$, then $f\sim_r D_\mu$ and $f$ is right simple.
\item[(ii)] If $r(f_3)=1$, $p=5$ and $6\leq \mu\leq 7$ then $f\sim_r E_\mu$ and $f$ is right simple.
\item[(iii)] If $r(f_3)=1$, $p>5$ and $6\leq \mu\leq 8$ then $f\sim_r E_\mu$ and $f$ is right simple.
\end{itemize} 
\end{proposition}

\begin{proof}
This may be proved in much the same way as \cite[Thm. I.2.51, Cor. I.2.52, Thm. I.2.53, Cor. I.2.54]{GLS06}, by applying the finite determinacy theorem in positive characteristic \cite[Thm. 2.1]{BGM12}. For details we refer to \cite[Prop. 4.3.5]{Ng13}.
\end{proof}

\begin{proposition}\label{pro1.5.3}
Let $mt(f)=3$. Let $r(f_3)$ be the number of linear factors of $f_3$. Then $f$ is not right simple if 
\begin{itemize}
\item[(i)]  either $p=3$;
\item[(ii)] or $p>3$, $r(f_3)\geq 2$ and $\mu\geq p$;
\item[(iii)] or $p>5$, $r(f_3)=1$ and $\mu>8$;
\item[(iv)] or $p=5$, $r(f_3)=1$ and $\mu\geq 8$.
\end{itemize} 
\end{proposition}

\begin{proof}
(i) We consider the unfolding $F(x,y,t)=f+t\cdot x^2$ of $f$ at $0$ over $\mathbb A^1$. Since $mt(f)=3$ and since $p=3$, it is easy to see that $\mu(f_t)>2$ for all $t\neq 0$. Proposition \ref{pro1.5.1}(ii) yields that $f_{t}, t\neq 0$, is not right simple and hence neither is $f$ by Proposition \ref{pro1.4.5.3}.

(ii) By Proposition \ref{pro1.5.2}(i), $f\sim_r x^2y+g(y)$ with $mt(g)=\mu-1$. It suffices to show that $h:=x^2y+g(y)$ is not right simple. We write 
$$g(y)=a\cdot y^{\mu-1}+\text{ higher trems}, a\neq 0$$
and consider the unfolding 
\begin{equation*}
h_t:=H(x,y,t)=
\begin{cases}
x^2y+g(y)+t x^2+ t y^p&\text{ if } \mu> p\\
x^2y+g(y)+at^2x^2+ 2at xy^{(p-1)/2}&\text{ if } \mu=p
\end{cases}
\end{equation*}
of $h$ at $0$ over $\mathbb A^1$. It is easy to see that $\mu(h_t)\geq p$ for all $t\neq 0$ (in fact, $\mu(h_t)= p$ for almost all $t$). It follows from Proposition \ref{pro1.5.1} that $h_t, t\neq 0$, is not right simple and hence neither is $h$ due to Proposition \ref{pro1.4.5.3}.

(iii) This is done by the same argument as in \cite[Thm. I.2.55(2)(ii)]{GLS06}.

(iv) Since $r(f_3)=1$ and $\mu\geq 8$, using the same argument as in  \cite[Thm. I.2.53]{GLS06}, we get
$$f\sim_r g:= x^2y+\alpha y^5+\beta xy^4+h(x,y)$$
with $\alpha, \beta \in K$ and $h\in \mathfrak{m}^6$. Consider the unfolding $g_t:=G(x,y,t)=g(x,y)+t\cdot xy^4$ of $g$ at $0$ over $\mathbb A^1$ and assume that $g_t\sim_r g_{t'}$, i.e. there exists an automorphism
\begin{eqnarray*}
\Phi : K[[x,y]]&\longrightarrow & K[[x,y]]\\
x&\mapsto & \varphi=\sum a_{ij} x^iy^j\\
y&\mapsto & \psi=\sum b_{ij} x^iy^j
\end{eqnarray*}
such that $g_{t}(x,y)=g_{t'}(\varphi,\psi)$. By a simple calculation we conclude that $(\beta+t)^3=(\beta+t')^3$ and hence, for fixed $t$, $g_t\sim_r g_{t'}$ for at most three values of $t'$. It follows from Corollary \ref{coro1.4.6} that $g$ is not right simple and hence neither is $f$.
\end{proof}

\begin{proof}[Proof of Theorem \ref{thm1.5.1}(i)]
It follows from Propositions \ref{pro1.4.4}, \ref{pro1.5.1}, \ref{pro1.5.2} and \ref{pro1.5.3}.
\end{proof}

\subsection{Right simple hypersurface singularities in characteristic $>2$}\label{sec3.3}
Our aim is to prove Theorem \ref{thm1.5.1}(ii). Let $f\in K[[{\bf x}]]=K[[x_1,\ldots,x_n]]$. We denote by
$$H(f):=\big( \frac{\partial^2 f}{\partial x_i\partial x_j}(0)\big)_{i,j=1,\ldots,n}\in \mathrm{Mat}(n\times n, K)$$
the {\em Hessian (matrix)} of $f$ and by $\mathrm{crk}(f):=n-\mathrm{rank}(H(f))$ the {\em corank} of $f$. 

\begin{lemma}[Right splitting lemma in characteristic different from 2] \label{lm1.5.3.0}
If $f\in \mathfrak{m}^2\subset K[[{\bf x}]], \mathrm{char}(K)>2,$ has corank $\mathrm{crk}(f)=k\geq 0$, then 
$$f\sim_r g(x_1,\ldots, x_k)+x_{k+1}^2+\ldots + x_{n}^2$$ 
with $g\in \mathfrak{m}^3$. $g$ is called the {\em residual part} of $f$, it is uniquely determined up to right equivalence.  
\end{lemma}

\begin{proof}
cf.  \cite[Thm. I.2.47]{GLS06}. The proof in \cite{GLS06} is given for $K=\Bbb C$ but works in characteristic different from 2.
\end{proof}

\begin{lemma}\label{lm1.5.3}
Let $p=\mathrm{char}(K)>2$ and let 
$$f_i(x_1,\ldots,x_n)=x_n^2+f'_i(x_1,\ldots, x_{n-1})\in\mathfrak{m}^2\subset  K[[x_1,\ldots,x_n]],\ i=1,2.$$
Then $f_1\sim_r f_2$ if and only if $f'_1\sim_r f'_2$.
\end{lemma}

\begin{proof}
The direction, $f'_1\sim_r f'_2\Rightarrow f_1\sim_r f_2$ is obvious. We now assume that $f_1\sim_r f_2$. Then $\mathrm{crk}(f_1)=\mathrm{crk}(f_2):=k$ and therefore $\mathrm{crk}(f'_1)=\mathrm{crk}(f'_2)=k$. It follows from Lemma \ref{lm1.5.3.0} that 
$$f'_i\sim_r g_i(x_1,\ldots, x_k)+x_{k+1}^2+\ldots + x_{n-1}^2,\ i=1,2.$$
and hence
$$f_i\sim_r g_i(x_1,\ldots, x_k)+x_{k+1}^2+\ldots + x_{n-1}^2+x_n^2.$$
This implies 
$$g_1(x_1,\ldots, x_k)+x_{k+1}^2+\ldots + x_n^2\sim_r g_2(x_1,\ldots, x_k)+x_{k+1}^2+\ldots + x_n^2$$
since $f_1\sim_r f_2$. The uniqueness of $g_i$ shows that $g_1\sim_r g_2$, i.e. there exists an automorphism $\Phi'\in Aut_K(K[[x_1,\ldots,x_k]])$ such that $\Phi'(g_1)=g_2$. The automorphism 
\begin{eqnarray*}
\Phi: K[[x_1,\ldots ,x_{n-1}]]&\longrightarrow & K[[x_1,\ldots ,x_{n-1}]]\\
x_i&\mapsto & \Phi'(x_i), i=1,\ldots, k\\
x_j &\mapsto & x_j, j=k+1,\ldots, n-1
\end{eqnarray*}
yields that 
$$g_1(x_1,\ldots, x_k)+x_{k+1}^2+\ldots + x_{n-1}^2\sim_r g_2(x_1,\ldots, x_k)+x_{k+1}^2+\ldots + x_{n-1}^2$$
and hence $f'_1\sim_r f'_2$. This completes the proof.
\end{proof}

\begin{lemma}\label{lm1.5.4}
Let $p=\mathrm{char}(K)>2, n\geq 2,$ and let 
$$f(x_1,\ldots,x_n)=x_n^2+f'(x_1,\ldots, x_{n-1})\in\mathfrak{m}^2\subset K[[x_1,\ldots,x_n]]$$ 
be such that $\mu(f)<\infty$. 
\begin{itemize}
\item[(i)] Let $F'({\bf x}',{t})\in \langle x_1,\ldots,x_{n-1}\rangle\subset  \mathcal{O}(T)[[x_1,\ldots,x_{n-1}]]$ be an unfolding of $f'$ at ${t}_0$ over an affine variety $T$ and let $F({\bf x},t)=x_n^2+F'({\bf x}',t)$. Then 
$$\mathcal{R}\text{-}\mathrm{mod}_F(f)=\mathcal{R}\text{-}\mathrm{mod}_{F'}(f').$$
\item[(ii)] We have
$$\mathcal{R}\text{-}\mathrm{mod}(f) \text{ in } K[[x_1,\ldots,x_n]]=\mathcal{R}\text{-}\mathrm{mod}(f') \text{ in } K[[x_1,\ldots,x_{n-1}]].$$
\end{itemize}
\end{lemma}

\begin{proof}
Let $k$ be sufficiently large for $f$ and for $f'$ w.r.t. $\mathcal{R}$. Let $\mathfrak{m}'$ be the maximal ideal in $K[[{\bf x}']]$ and let $J'_k:=K[[{\bf x}']]/\mathfrak{m}'^{k+1}$.

(i) Consider the morphisms
\begin{displaymath}
\begin{array}{rcccrcc}
h_k:T&\longrightarrow & J_k &\text{ and } &h'_k:T&\longrightarrow & J'_k\\
t&\mapsto & j^k f_t & &t&\mapsto & j^k f'_t .
\end{array}
\end{displaymath}
and $p:J_k\to J'_k$ the natural projection. Then $h'_k=p\circ h_k$ and 
$$h_k^{-1}(\mathcal{R}\cdot h_k(t))=h'^{-1}_k(\mathcal{R}\cdot h'_k(t))$$ by Lemma \ref{lm1.5.3}. It follows from Corollary \ref{coro1.4.3} that $\mathcal{R}\text{-}\mathrm{mod}_{h_k}(t_0)=\mathcal{R}\text{-}\mathrm{mod}_{h'_k}(t_0),$ and hence $\mathcal{R}\text{-}\mathrm{mod}_F(f)=\mathcal{R}\text{-}\mathrm{mod}_{F'}(f').$

(ii) Let $\{g'_1({\bf x}'),\ldots,g'_l({\bf x}') \}$ be a system of generators of $\mathfrak{m}'/\mathfrak{m}'\cdot j(f')$. Then $\{x_n,g_1({\bf x}),\ldots,g_l({\bf x}) \}$ with $g_i({\bf x})=g'_i({\bf x}')$, is a system of generators of $\mathfrak{m}/\mathfrak{m}\cdot j(f)$.  Proposition \ref{pro1.4.4.1} yields that 
$$F'({\bf x}',t)=f'+\sum_{i=1}^l t_i g'_i({\bf x}') \text{ (resp. } F_1({\bf x},t)= f+\sum_{i=1}^l t_i g_i({\bf x})+t_{l+1} x_{n})$$ 
is a right complete unfoldings of  $f'$ (resp. of $f$) over $\mathbb A^{l}$ (resp. $\mathbb A^{l+1}$), i.e.
$$\mathcal{R}\text{-}\mathrm{mod}(f)=\mathcal{R}\text{-}\mathrm{mod}_{F_1}(f)\ (\text{resp. } \mathcal{R}\text{-}\mathrm{mod}(f') \text{ in } K[[{\bf x}']]=\mathcal{R}\text{-}\mathrm{mod}_{F'}(f'))$$
due to Proposition \ref{pro1.4.5}. Note that $F_1({\bf x},t)\sim_r x_1$ and therefore $\mathcal{R}\text{-}\mathrm{mod}(F_1({\bf x},t))=0$ for all $t=(t_1,\ldots,t_{l+1})\in \mathbb A^{l+1}$ with $t_{l+1}\neq 0$. 

Consider the inclusion $\mathbb A^{l}\subset \mathbb A^{l+1}, t=(t_1,\ldots,t_{l})\mapsto (t_1,\ldots,t_{l},0)$ and the unfolding 
$$F=f+\sum_{i=1}^l t_i g_i({\bf x})$$ 
of $f$ at $0$ over $\mathbb A^{l}$. Since $\mathcal{R}\text{-}\mathrm{mod}(F_1({\bf x},t))=0$ for all $t\in \mathbb A^{l+1}\setminus \mathbb A^{l}$, using the same argument as in the proof of Proposition \ref{pro1.4.1}(iii) (see also \cite[Prop. 3.2.6]{Ng13}) we obtain that $\mathcal{R}\text{-}\mathrm{mod}_{F_1}(f)=\mathcal{R}\text{-}\mathrm{mod}_{F}(f)$. This implies, by (i), that $$\mathcal{R}\text{-}\mathrm{mod}_{F_1}(f)=\mathcal{R}\text{-}\mathrm{mod}_{F}(f)=\mathcal{R}\text{-}\mathrm{mod}_{F'}(f')$$
and hence $\mathcal{R}\text{-}\mathrm{mod}(f)=\mathcal{R}\text{-}\mathrm{mod}(f')$.
\end{proof}

\begin{proof}[Proof of Theorem \ref{thm1.5.1}(ii)]
The ``if"-statement follows from Theorem \ref{thm1.5.1}(i) and Lemma \ref{lm1.5.4}. We now consider any simple singularity $f\in\mathfrak{m}^2\subset K[[{\bf x}]]$. Then, by the splitting lemma, 
$$f\sim_r f'(x_1,\ldots, x_k)+x_{k+1}^2+\ldots + x_{n}^2$$ 
with $f'\in \langle x_1,\ldots, x_k \rangle^3$ and $k=\mathrm{crk}(f)$. Again by Lemma \ref{lm1.5.4}, 
$$\mathcal{R}\text{-}\mathrm{mod}(f')=\mathcal{R}\text{-}\mathrm{mod}(f)=0.$$
It follows from Proposition \ref{pro1.4.4} that
$$0=\mathcal{R}\text{-}\mathrm{mod}(f') \geq \binom{m+k-1}{m}-k^2,$$
where $m=mt(f')\geq 3$. This implies that $k\leq 2$, i.e.
 $$f\sim_r g(x_1,x_2)+x_{3}^2+\ldots + x_{n}^2,$$
for some simple singularity $g\in K[[x_1,x_2]]$. The proof thus follows from Theorem \ref{thm1.5.0}, \ref{thm1.5.1}(i) and Lemma \ref{lm1.5.3}.
\end{proof}

\subsection{Right simple hypersurface singularities in characteristic 2}
Let $p=\mathrm{char}(K)=2$ and let $n\geq 2$.

\begin{lemma}[Right splitting lemma in characteristic 2]\label{lm1.5.5}
Let $f\in \mathfrak{m}^2\subset K[[x_1,\ldots,x_n]], n\geq 2$. Then there exists an $l$, $0\leq 2l\leq n$ such that
$$f\sim_r x_1x_2+x_3x_4+\ldots+x_{2l-1}x_{2l}+g(x_{2l+1},\ldots,x_n)$$
with $g\in \langle x_{2l+1},\ldots,x_{n} \rangle^3$ or $g\in x_{2l+1}^2+\langle x_{2l+1},\ldots,x_{n} \rangle^3$ if $2l<n$. $g$ is called the {\em residual part} of $f$, it is uniquely determined up to right equivalence.
\end{lemma}

\begin{proof}
It follows easily from \cite[Lemmas 1 and 2]{GK90}.
\end{proof}

\begin{lemma}\label{lm1.5.6}
Let $\mu(f)<\infty$ and
$$f(x_1,\ldots,x_n)=x_{n-1}x_n+f'(x_1,\ldots, x_{n-2})\in\mathfrak{m}^2\subset K[[x_1,\ldots,x_n]].$$
Then 
$$\mathcal{R}\text{-}\mathrm{mod}(f) \text{ in } K[[x_1,\ldots,x_n]]=\mathcal{R}\text{-}\mathrm{mod}(f') \text{ in } K[[x_1,\ldots,x_{n-2}]].$$
\end{lemma}

\begin{proof}
By using the same arguments as in the proof of Lemma \ref{lm1.5.4}.
\end{proof}

\begin{remark}{\rm
Since $\mu(x_1x_2)=1$, $x_1x_2\in K[[x_1,x_2]]$ is right 2-determined and any unfolding of $x_1x_2$ is either right equivalent to itself or smooth. Hence $x_1x_2$ is right simple.
}\end{remark}

\begin{proposition}\label{pro1.5.8}
Let $f\in \mathfrak{m}^2\subset K[[{\bf x}]]$ with $\mu(f)<\infty$. Then $f$ is not right simple if
\begin{itemize}
\item[(i)] either $f=x_1^2+g(x_1, \ldots, x_n)\in K[[{\bf x}]]$ with $g\in\mathfrak{m}^3$,
\item[(ii)] or $f\in \mathfrak{m}^3$.
\end{itemize}
\end{proposition}

\begin{proof}
(i) Let $k\geq 3$ be sufficiently large for $f$ w.r.t. $\mathcal{R}$ and let $X:=\mathfrak{m}^2/\mathfrak{m}^{k+1}$. Then $\mathcal{R}\text{-}\mathrm{mod}(f)=\mathcal{R}_k\text{-}\mathrm{mod}(j^k f)$. Let 
$$Y:=x_1^2+\mathfrak{m}^3/\mathfrak{m}^{k+1}\subset J_k,\ Y':=x_1^2+\mathfrak{m}^3/\mathfrak{m}^{4}$$
and let 
$$\mathcal{H}:=\{\Phi\in \mathcal{R}_k\ | \Phi(x_1)=x_1\},\ \mathcal{H}':=\{\Phi\in \mathcal{R}_1\ | \Phi(x_1)=x_1\}.$$
Then $\mathcal{H}$ (resp. $\mathcal{H}'$) acts on $Y$ (resp. $Y'$) by $(\Phi,y)\mapsto \Phi(y)$ and we have
$$i^{-1}(\mathcal{R}_k\cdot i(y))=\mathcal{H}\cdot y \subset p^{-1}(\mathcal{H}'\cdot p(y)) \ \forall y\in Y'$$
with the inclusion $i:Y\hookrightarrow X$ and the projection $p:Y\twoheadrightarrow Y'$. It follows from Proposition \ref{pro1.4.1}(iv) that 
$$\mathcal{R}_k\text{-}\mathrm{mod}(y)\geq \mathcal{H}\text{-}\mathrm{mod}(y) \geq \mathcal{H}'\text{-}\mathrm{mod}(p(y)),\forall y\in Y.$$ 
Moreover, Proposition \ref{pro1.4.1}(i) yields that
$$\mathcal{H}'\text{-}\mathrm{mod}(p(y))\geq \dim Y'-\dim \mathcal{H}'= \binom{n+2}{3}-n(n-1)\geq 1.$$
This implies that $\mathcal{R}_k\text{-}\mathrm{mod}(y)\geq 1$ for all $y\in Y$ and hence $\mathcal{R}\text{-}\mathrm{mod}(f)\geq 1$.

(ii) By (i), $f_t$ is not right simple for all $t\neq 0$, where $f_t({\bf x}):=f({\bf x})+t x_1^2$ is an unfolding of $f$ at $0$ over $\mathbb A^1$. Hence Proposition \ref{pro1.4.5.3} yields that  $f$ is not right simple.
\end{proof}

\begin{proof}[Proof of Theorem \ref{thm1.5.2}]
The ``if"-statement is obvious. Now, take a right simple singularity $f\in \mathfrak{m}^2\subset K[[{\bf x}]]$. Then $mt(f)=2$ by Proposition \ref{pro1.4.4}. The splitting lemma (Lemma \ref{lm1.5.5}) yields that $f$ is right equivalent to 
$$x_1x_2+x_3x_4+\ldots+x_{2l-1}x_{2l}+g(x_{2l+1},\ldots,x_n)$$
with $g\in \langle x_{2l+1},\ldots,x_{n} \rangle^3$ or $g\in x_{2l+1}^2+\langle x_{2l+1},\ldots,x_{n} \rangle^3$ if $2l<n$. Combining Lemma \ref{lm1.5.6} and Proposition \ref{pro1.5.8} we obtain that $2l=n$, which proves the theorem.
\end{proof}

\appendix
\section{}
\subsection{Modality for algebraic group actions}\label{sec.a1} Let an algebraic group $G$ act on the variety $X$. We define the notion of number of $G$-parameters and show that it coincides with the $G$-modality, which proves the independence of modality of the Rosenicht stratification. Moreover if $X'$ is any variety (without $G$-action) and if $h:X'\to X$ is a morphism, we generalize these notions to the equivalence relation induced by $h$ on $X'$. This allows us to use deformation theory. As consequences, we give interesting properties of modality, which are used in Section \ref{sec2}.

\begin{definition}{\rm
Let $U\subset X$ be an open neighbourhood of $x\in X$ and $W$ be constructible in $X$. We introduce
\begin{eqnarray*}
\dim_x W&:=&\max\{\dim Z\ |\ Z \text{ an irreducible component of }W \text{ containing }x \},\\
U(i)&:=&U_G(i):=\{ y \in U\ |\ \dim_y (U\cap G\cdot y ) =i\}, i\geq 0,\\
G\text{-}\mathrm{par}(U)&:=&\max_{i\geq 0}\{\dim U(i)-i\},
\end{eqnarray*}
and call
$$G\text{-}\mathrm{par}(x):=\min \{G\text{-}\mathrm{par}(U)\ |\ U \text{ a neighbourhood of x}\}$$
the {\em number of $G$-parameters} of $x$ (in $X$).
}\end{definition} 

The following proposition is a special case of Proposition \ref{pro1.4.2} (with $h=\mathrm{id}$), which is proven below.

\begin{proposition}\label{pro1.4.0}
We have $G\text{-}\mathrm{par}(U)=G\text{-}\mathrm{mod}(U)$ and therefore $G\text{-}\mathrm{par}(x)=G\text{-}\mathrm{mod}(x)$ for all $x\in X$.
\end{proposition} 

\begin{corollary}\label{cor1.4.1}
$G\text{-}\mathrm{mod}(U)$ and $G\text{-}\mathrm{mod}(x)$ are independent of the Rosenlicht stratification of $X$.
\end{corollary}

We call $x,y\in X$ {\em $G$-equivalent}, denoted by $x\sim_G y$, if they ly in the same $G$-orbit.

\begin{proposition}\label{pro1.4.1}
Let the algebraic group $G$ act on the variety $X$. 
\begin{itemize}
\item[(i)] If $G$ and $X$ are both irreducible then $G\text{-}\mathrm{mod}(x)\geq \dim X-\dim G$.
\item[(ii)] For any $y\sim_G x$, $G\text{-}\mathrm{mod}(x)=G\text{-}\mathrm{mod}(y)$.
\item[(iii)] If the subvariety $X'\subset X$ is invariant under $G$ and if $x\in X'$ then
$$G\text{-}\mathrm{mod}(x)\text{ in } X \geq G\text{-}\mathrm{mod}(x) \text{ in } X'.$$ 
Equality holds if $G\text{-}\mathrm{mod}(x) \text{ in } X'\geq G\text{-}\mathrm{mod}(y),\ \forall y\in X\setminus X'.$
\item[(iv)] Let the algebraic group $G'$ act on the variety $X'$ and let $p:X\rightarrow X'$ be a morphism of varieties.
\begin{itemize}
\item[(1)]  If $p$ is open and if 
$$G\cdot x\subset p^{-1}\big ( G'\cdot p(x)\big ),\ \forall x\in X,$$
then
$$G\text{-}\mathrm{mod}(x)\geq G'\text{-}\mathrm{mod}(p(x)), \ \forall x\in X.$$
\item[(2)] If $p$ is {\em equivariant} (i.e. $G\cdot x = p^{-1}\big ( G'\cdot p(x)\big ),\ \forall x\in X$), then 
 $$G\text{-}\mathrm{mod}(x)\leq G'\text{-}\mathrm{mod}(p(x)), \ \forall x\in X.$$
\end{itemize}
In particular, if $p$ is open and equivariant then $G\text{-}\mathrm{mod}(x)=G'\text{-}\mathrm{mod}(p(x)), \ \forall x\in X.$
\end{itemize}
\end{proposition}

For the elementary but not so short proof, using $G\text{-}\mathrm{mod}=G\text{-}\mathrm{par}$, we refer to \cite[Prop. 3.2.4--3.2.7]{Ng13}.

In order to relate the notion of modality to deformation theory we introduce the notion of $G$-modality w.r.t. a morphism from an arbitrary variety to the $G$-variety $X$.

\begin{definition}\label{def1.4.2.2}{\rm 
Let $\{X_{j},j=1,\ldots,s\}$ be a Rosenlicht stratification of $X$ under $G$ with projections $p_j:X_j\to X_j/G$. 

Let $h: X'\to X$ be a morphism of algebraic varieties, and let $U'$ be an open neighbourhood of $x'\in X'$. Set $X'_j:= h^{-1}(X_j)$, $U'_j:=U'\cap X'_j$. We define
$$G\text{-}\mathrm{mod}_h(U'):=\max_{j=1,\ldots, s}\{\dim \big( p_j(h(U'_j))\big) \},$$
and call 
$$G\text{-}\mathrm{mod}_{h}(x'):=\min \{G\text{-}\mathrm{mod}(U')\  |\ U' \text{ a neighbourhood of  } x'\text{ in } X'\}$$
the {\em $G$-modality} of $x'$ w.r.t. $h$.
}\end{definition}

\begin{definition}\label{def1.4.2.1}{\rm
Let the  algebraic group $G$ act on the variety $X$, let $h: X'\to X$ be a morphism of algebraic varieties and let $U'$ be an open neighbourhood of $x'\in X'$. For $u'\in U'$ and each $i\geq 0$ we define
\begin{eqnarray*}
V_{G,h}(u')&:=&\{ v'\in X'\ |\ G\cdot h(v')=G\cdot h(u')\}=h^{-1}(G\cdot h(u')),\\
U'(i)&:=&U'_{G,h}(i):=\{u'\in U'\ |\ \dim_{u'} \big( U'\cap V_{G,h}(u')\big)=i\},\\
G\text{-}\mathrm{par}_h(U')&:=&\max_{i\geq 0}\{\dim U'_{G,h}(i)-i\},
\end{eqnarray*}and call
$$G\text{-}\mathrm{par}_h(x'):=\min\{G\text{-}\mathrm{par}_h(U')\ |\ U'\text{ a neighbourhood of }x'\in X'\}$$
{\em the number of $G$-parameters of $x'$ w.r.t. $h$}.
}\end{definition}

If we call $x',y'\in X'$ equivalent if $h(x')$ and $h(y')$ ly in the same orbit in $X$, we get an equivalence relation on $X'$ with equivalence class of $x'$ equal to
$$V_{G,h}(x')=h^{-1}\big(G\cdot h(x')\big).$$
It follows in particular that each equivalence class is locally closed in $X'$. 

\begin{proposition}\label{pro1.4.2}
We have $G\text{-}\mathrm{par}_{h}(U')=G\text{-}\mathrm{mod}_{h}(U')$ and therefore $G\text{-}\mathrm{par}_{h}(x')=G\text{-}\mathrm{mod}_{h}(x')$ for all $x'\in X'$.
\end{proposition} 

For the proof we need the following properties of fibers of a morphism (\cite{Mum88}, \cite{Bor91}). Let $f : X \to Y$ be a morphism of (not necessarily irreducible) algebraic varieties. First, Chevalley's theorem says that if $W\subset X$ is constructible in $X$, then $f(W)$ is constructible in $Y$. Secondly, the function $x\mapsto e(x):=\dim_x f^{-1}(f(x))$ is {\em upper semi-continuous}, i.e. for all integers $n$ the set $\{x\in X\ |\ e(x) \geq n\}$ is closed. For the proof of these statements and the following lemma some well known results in \cite[I.8]{Mum88} for irreducible varieties, are used. For details see \cite[Cor. 3.1.7 and 3.1.8]{Ng13}.

\begin{lemma}\label{lm1.4.1}
Let $f : X \to Y$ be a morphism algebraic varieties and let $e:X\to \mathbb N$ be the function defined by $x\mapsto e(x):=\dim_x f^{-1}(f(x))$.
\begin{itemize}
\item[(i)] If $e(x)$ is constant, say $e(x)=i$ for all $x\in X$, then
$$\dim X = i+\dim f(X).$$
\item[(ii)] We have
$$\max_{i\geq 0}\{ \dim \big( e^{-1}(i)\big)-i\}\geq \dim f(X).$$
\end{itemize}
\end{lemma}

\begin{proof}[Proof of Proposition \ref{pro1.4.2}]
We use the notations of Definition \ref{def1.4.2.1} and consider the composition
$$h_j:U'_j\overset{h}{\longrightarrow } X_j \overset{p_j}{\longrightarrow } X_j/G.$$
Note that 
$$h^{-1}_j(h_j(x'))=U'_j\cap h^{-1}((G\cdot h(x')))=U'_j\cap V_{G,h}(x'),\ \forall x'\in U'_j.$$ 
By the upper semi-continuity of the functions $e_j: U'_j\to \mathbb N, x' \mapsto \dim_{x'} h^{-1}_j(h_j(x'))$, the sets $e_j^{-1}(i)$ are locally closed and then
$$U'(i)=\bigcup_{j=1}^{s} e_j^{-1}(i)$$
is constructible in $X'$ for all $i\geq 0$. 

Taking an $i$ such that $G\text{-}\mathrm{par}_h(U')=\dim U'(i)-i$ and applying Lemma \ref{lm1.4.1}(i) we deduce
$$G\text{-}\mathrm{par}_h(U')=\max_j\{\dim e_j^{-1}(i)-i \}=\max_j\{\dim \big (h_j( e_j^{-1}(i)) \big)\}\leq G\text{-}\mathrm{mod}_h(U').$$

Let $j\in \{1,\ldots, s\}$ be such that 
$$G\text{-}\mathrm{mod}_h(U')=\dim \big( p_j(h(U'_j))\big)=\dim h_j(U'_{j}).$$
Then
$$G\text{-}\mathrm{par}_h(U')=\max_i\{\dim U'(i)-i\} \geq \max_i\{\dim(e_j^{-1}(i))- i\}\geq \dim h_j(X'_j)=G\text{-}\mathrm{mod}_h(U'),$$
where the second inequality follows from Lemma \ref{lm1.4.1}(ii). Hence $G\text{-}\mathrm{par}_h(U')=G\text{-}\mathrm{mod}_h(U').$
\end{proof}

Let the  algebraic group $G$ act on the variety $X$. The proofs of the following corollaries follow easily from Definition \ref{def1.4.2.1}, \ref{def1.4.2.2} and Proposition \ref{pro1.4.2} (for details we refer to \cite[Cor. 3.3.4-3.3.7]{Ng13}).

\begin{corollary}\label{coro1.4.1}
Let $h: X'\to X$ be a morphism of algebraic varieties. Then
$$G\text{-}\mathrm{mod}_h(x')\leq G\text{-}\mathrm{mod}(h(x')), \ \forall x'\in X'.$$
Equality holds if for every open neighbourhood $U'$ of $x'$, there exists an open neighbourhood $U$ of $h(x')$ in $X$ s.t. $U\subset h(U')$. In particular, equality holds if $h$ is open.
\end{corollary}

\begin{corollary}\label{coro1.4.2}
Let
$$g:X''\overset{h'}{\to}X' \overset{h}{\to} X$$
be morphisms of algebraic varieties. Then
$$G\text{-}\mathrm{mod}_g(x'')\leq G\text{-}\mathrm{mod}_h(h'(x'')), \ \forall x''\in X''.$$
Equality holds if for every open neighbourhood $U''$ of $x''$, there exists an open neighbourhood $U'$ of $h'(x'')$ in $X'$ s.t. $U'\subset h'(U'')$. In particular, equality holds if $h'$ is open.
\end{corollary}

\begin{corollary}\label{coro1.4.3}
Let the algebraic groups $G$ resp. $G'$ act on the varieties $X$ resp. $X'$. Let 
$$h: Y\to X \text{ and } h':Y\to X'$$ be two morphisms of varieties such that 
$$h^{-1}(G\cdot h(y))=h'^{-1}(G'\cdot h' (y)), \forall y\in Y.$$
Then for any open subset $V\subset Y$ we have $G\text{-}\mathrm{mod}_h(V)=G'\text{-}\mathrm{mod}_{h'}(V)$. Consequently
$$G\text{-}\mathrm{mod}_h(y)=G'\text{-}\mathrm{mod}_{h'}(y), \ \forall y\in Y.$$
\end{corollary}

\begin{corollary}\label{coro1.4.4}
Let $h: X'\to X $ and $h_i: Y_i\to X, i=1,\ldots,k$, be morphisms of varieties and let $U'$ be open in $X'$ satisfying, that for all $x'\in U'$ there exists $y_i\in Y_i$ such that $h(x')\sim_G h_i(y_i)$. Then
\begin{itemize}
\item[(i)] $G\text{-}\mathrm{mod}_h(U')\leq \max\{ G\text{-}\mathrm{mod}_{h_i}(Y_i)| i=1,\ldots,k\}$.
\item[(ii)] Assume moreover that for all $y_i\in Y_i$ there exists $x'\in X'$ such that $h(x')\sim_G h_i(y_i)$ then
$$G\text{-}\mathrm{mod}_h(U') = \max\{ G\text{-}\mathrm{mod}_{h_i}(Y_i)| i=1,\ldots,k\}.$$
\end{itemize}
\end{corollary}

\subsection{Right and contact groups}\label{sec.a2}
The group $\mathcal{R}:=Aut_K (K[[{\bf x}]])$ of automorphisms of the local analytic $K$-algebra $K[[{\bf x}]]=K[[x_1,\ldots,x_n]]$ is called the {\em right group}. The {\em contact group} $\mathcal{K}$ is the semi-direct product of $Aut_K (K[[{\bf x}]])$ with the group $(K[[{\bf x}]])^{*}$ of units in $K[[{\bf x}]]$. These groups act on $K[[{\bf x}]]$ by
\begin{displaymath}
\begin{array}{ccccccc}
\mathcal{R}\times K[[{\bf x}]]&\longrightarrow & K[[{\bf x}]] &\text{ and } &\mathcal{K}\times K[[{\bf x}]]&\longrightarrow & K[[{\bf x}]]\\
(\Phi,f) &\mapsto &\Phi(f) & &((\Phi,u),f) &\mapsto &u\cdot\Phi(f).
\end{array}
\end{displaymath}
Two elements $f,g\in K[[{\bf x}]]$ are called {\em right} ($\sim_r$) resp. {\em contact} ($\sim_c$) {\em equivalent} iff they belong to the same $\mathcal{R}$- resp. $\mathcal{K}$-orbit.

Note that neither $\mathcal{R}$ nor $\mathcal{K}$ are algebraic groups, as they are infinite dimensional. In order to be able to apply the results from the previous section, we have to pass to the jet spaces.

An element $\Phi$ in the right group $\mathcal{R}$ is uniquely determined by $n$ power series
$$\varphi_i:=\Phi(x_i)=\sum_{j=1}^n a_i^j x_j+\text{ terms of higher oder}$$
such that $\det (a_i^j)\neq 0$. For each integer $k$ we define the $k$-jet of $\Phi$,
$$\Phi_k:=(j^k(\varphi_1),\ldots,j^k(\varphi_n)).$$
Here  
$$j^k:K[[{\bf x}]]\to J_k:=K[[{\bf x}]]/\mathfrak{m}^{k+1}$$
denotes the canonical projection to the $k$-jet space $J_k$ and $\mathfrak{m}=\langle {\bf x}\rangle$ the maximal ideal of $K[[{\bf x}]]$. For $f\in K[[{\bf x}]]$, $j^kf:=j^k(f)$ is represented in $K[[{\bf x}]]$ by the power series expansion of $f$ up to order $k$. We call 
$$\mathcal{R}_k:=\{\Phi_k\ |\ \Phi\in  \mathcal{R}\}\text{ respectively } \mathcal{K}_k:=\{(\Phi_k,j^k u)\ |\ (\Phi,u)\in  \mathcal{K}\}$$
the $k$-jet of $\mathcal{R}$ respectively of $\mathcal{K}$. Note that $J_k$ is an affine space and $\mathcal{R}_k, \mathcal{K}_k$ are affine algebraic groups, all equipped with the Zariski topology. These groups act algebraically on $J_k$ by
\begin{displaymath}
\begin{array}{ccccccc}
\mathcal{K}_k\times J_k&\longrightarrow & J_k &\text{ and } & \mathcal{K}_k\times J_k&\longrightarrow & J_k\\
(\Phi_k,j^k f)&\mapsto &j^k(\Phi_k(j^kf)) & &((\Phi_k,j^k u),j^k f)&\mapsto &j^k(j^ku\cdot\Phi_k(j^kf)).
\end{array}
\end{displaymath}
$$$$

The first statement (i) of the following lemma says that the Milnor number $\mu$ and the Tjurina number $\tau$ are semi-continuous w.r.t. unfoldings. Its proof can be adapted from the construction in \cite[Thm. I.2.6]{GLS06} by applying \cite[Thm. 12.8]{Har77}. The second statement follows from (i).

In order to shorten notations we define a {\em topology} on $K[[{\bf x}]]$ to be the coarsest topology such that all projections $j^k:K[[{\bf x}]]\to J_k$ are continuous, i.e. $V\subset K[[{\bf x}]]$ is open iff $j^k(V)$ is open in $J_k$ for all $k$. 

\begin{lemma}[Semi-continuity of $\mu$ and $\tau$]\label{pro1.1.1}
Let $f\in K[[{\bf x}]]$ with $\mu(f)<\infty$ resp. $\tau(f)<\infty$.
\begin{itemize}
\item[(i)] Let $f_{t}({\bf x})=F({\bf x},t)$ be an unfolding of $f=f_{t_0}$ at $t_0$ over an affine variety $T$. Then there exists an open neighbourhood $U\subset T$ of ${t}_0$ such that 
$$\mu(f_{t})\leq \mu(f), \text{ resp. } \tau(f_{\bf t})\leq \tau(f), \forall {t}\in U.$$
More general, for all $i\in \mathbb N$ the sets
$$\{t\in T\ |\ \mu(f_{t})\leq i\} \text{ resp. } \{t\in T\ |\ \tau(f_{t})\leq i\}$$
are open in $T$.
\item[(ii)] The functions $\mu \text{ and } \tau$ are upper semi-continuous on $K[[{\bf x}]]$, i.e. for all $i\in \mathbb N$, the sets
$$U_{\mu,i}:=\{f\in K[[{\bf x}]]\ |\ \mu(f)\leq i\} \text{ and } U_{\tau,i}:=\{f\in K[[{\bf x}]]\ |\ \tau(f)\leq i\}$$ 
are open in $K[[{\bf x}]]$.
\end{itemize}
\end{lemma}

\addtocontents{toc}{\protect  \vskip 0.2cm }

\end{document}